\RequirePackage{etex}
\documentclass{amsart}
\usepackage{amsmath}
\usepackage{amsthm}
\usepackage{booktabs}
\usepackage{enumerate}
\usepackage{color}
\usepackage{float}
\usepackage{graphicx}
\usepackage{amssymb}
\usepackage{latexsym}
\usepackage{amsfonts}
\usepackage{tabularx}
\usepackage{cite}
\usepackage{hyperref}
\usepackage{mathtools}
\usepackage[all]{xy}
\usepackage{tikz, tikz-3dplot, tkz-graph}
\usepackage{tikz-cd}
\usetikzlibrary{decorations.markings}
\tikzstyle{vertex}=[circle, draw, inner sep=0pt, minimum size=13pt]

\allowdisplaybreaks

\hypersetup{unicode=false, pdftoolbar=true, pdfmenubar=true,
        pdffitwindow=false, pdfstartview={FitH}, pdfnewwindow=true,
colorlinks=false, linkcolor=blue, citecolor=green}

\newtheorem{theorem}{Theorem}[section]

\newtheorem{proposition}[theorem]{Proposition}

\newtheorem{example}[theorem]{Example}
\newtheorem{lemma}[theorem]{Lemma}
\theoremstyle{definition}

\newtheorem{definition/notation}[theorem]{Definition/Notation}

\numberwithin{equation}{section}

\begin{document}

\title[Quantized nilradicals]{ Quantized nilradicals of parabolic subalgebras
        of $\mathfrak{sl}(n)$ and algebras of coinvariants}

\author{Andrew Jaramillo and Garrett Johnson}

\address{Department of Mathematics and Physics \\North Carolina Central
University\\Durham, NC \\USA}
\email{gjohns62@nccu.edu}

\address{Department of Mathematics\\University of Connecticut\\Hartford, CT \\USA}
\email{andrew.jaramillo@uconn.edu}



\thanks{The second author was supported in part by NSF grant DMS-1900823.}

\begin{abstract}

    Let $P_J$ be the standard parabolic subgroup of $SL_n$ obtained by deleting
    a subset $J$ of negative simple roots, and let $P_J = L_JU_J$ be the
    standard Levi decomposition.  Following work of the first author, we study
    the quantum analogue $\theta: {\mathcal O}_q(P_J) \to{\mathcal O}_q(L_J)
    \otimes {\mathcal O}_q(P_J)$ of an induced coaction and the corresponding
    subalgebra ${\mathcal O}_q(P_J)^{\operatorname{co} \theta} \subseteq
    {\mathcal O}_q(P_J)$ of coinvariants. It was shown that the smash product
    algebra ${\mathcal O}_q(L_J)\# {\mathcal O}_q(P_J)^{\operatorname{co}
    \theta}$ is isomorphic to ${\mathcal O}_q(P_J)$. In view of this,
    ${\mathcal O}_q(P_J)^{\operatorname{co} \theta}$ --  while it is not a Hopf
    algebra -- can be viewed as a quantum analogue of the coordinate ring
    ${\mathcal O}(U_J)$.

    In this paper we prove that when $q\in \mathbb{K}$ is nonzero and not a
    root of unity, ${\mathcal O}_q(P_J)^{\operatorname{co} \theta}$ is
    isomorphic to a quantum Schubert cell algebra ${\mathcal U}_q^+[w]$
    associated to a parabolic element $w$ in the Weyl group of
    $\mathfrak{sl}(n)$. An explicit presentation in terms of generators and
    relations is found for these quantum Schubert cells.

\end{abstract}

\maketitle

\section{Introduction and overview of results in the paper}

Let $SL_n$ be the complex algebraic group of $n \times n$ matrices having
determinant equal to one, and let $P_J$ be the standard parabolic subgroup of
block upper triangular matrices of $SL_n$ obtained by deleting a subset $J$ of
negative simple roots of $SL_n$. The group $P_J$ admits a Levi decomposition
$P_J = L_JU_J$, where $L_J$ is the standard Levi factor of block diagonal
matrices in $P_J$, and $U_J$ is the unipotent subgroup of matrices in $P_J$
having identity matrices along the block diagonal. Multiplication $L_J \times
P_J \to P_J$ induces a coaction, ${\mathcal O}(P_J) \to {\mathcal O}(L_J)
\otimes {\mathcal O}(P_J)$, where ${\mathcal O}(L_J)$ and ${\mathcal O}(P_J)$
are the coordinate rings of $L_J$ and $P_J$ respectively.

With the classical case in mind, we turn our attention to the corresponding
quantized coordinated rings, ${\mathcal O}_q(L_J)$ and ${\mathcal O}_q(P_J)$.
Here and below, the base field for all algebras is an arbitrary field
$\mathbb{K}$ that contains a nonzero element $q\in \mathbb{K}$ that is not a
root of unity. Define
\[
    \widehat{q} := q - q^{-1}.
\]
Following \cite{Jaramillo 1, Jaramillo 2}, we focus on the quantum analogue of
the coaction above,
\[
    \theta: {\mathcal O}_q(P_J) \to{\mathcal O}_q(L_J) \otimes {\mathcal
    O}_q(P_J).
\]
An element $x \in {\mathcal O}_q(P_J)$ is a \textit{(left) coinvariant} if
$\theta(x) = 1 \otimes x$. It was shown in \cite[Theorems 3.46 and
3.49]{Jaramillo 2} that the subalgebra of coinvariants ${\mathcal
O}_q(P_J)^{\operatorname{co} \theta} \subseteq {\mathcal O}_q(P_J)$ has a
presentation as an iterated Ore extension $\mathbb{K}[t_1][t_2; \tau_2,
\delta_2]\cdots [t_M; \tau_M, \delta_M]$, where $M = \operatorname{dim}(U_J)$
and is, in fact, a Cauchon-Goodearl-Letzter extension.  It was also shown that
the smash product algebra ${\mathcal O}_q(L_J) \# {\mathcal
O}_q(P_J)^{\operatorname{co} \theta}$ is isomorphic as a $\mathbb{K}$-algebra
to ${\mathcal O}_q(P_J)$ \cite[Theorem 3.19]{Jaramillo 2}.  In view of this,
${\mathcal O}_q(P_J)^{\operatorname{co} \theta}$ -- while it is not a Hopf
algebra -- can be viewed as a quantized version of the coordinate ring
${\mathcal O}(U_J)$.  In fact, if $q$ is put equal to $1$ in the defining
relations of ${\mathcal O}_q(P_J)^{\operatorname{co} \theta}$, we recover the
defining relations of ${\mathcal O}(U_J)$.

The generators of ${\mathcal O}_q(P_J)^{\operatorname{co} \theta}$ can be
indexed by elements of the set
\[
    \Phi_J := \{ (i, j) \in \{1,\dots, n\} \times \{1,\dots, n\} \mid
    \exists k\in J \text{ such that } i \leq k < j\}.
\]
We will denote the generators of ${\mathcal O}_q(P_J)^{\operatorname{co}
\theta}$ by $u_{ij}$, ($(i, j) \in \Phi_J$). Each $u_{ij}$ is a certain ratio
of quantum minors in the quantized coordinate ring ${\mathcal O}_q(P_J)$ (see
\ref{ratio of quantum minors} in Section \ref{algebra of coinvariants,
section}).  Viewing $u_{ij}$ as occupying the ($i, j$)-position in an $n\times
n$ array, we observe that the full set of generators of ${\mathcal
O}_q(P_J)^{\operatorname{co} \theta}$ forms a block upper triangular shape that
depends on $J$.  We define the function $r: \{1,\dots, n\} \to J \cup \{0 \}$,
\[
    r(m) := \operatorname{max}\{k \in J \cup \{0\} \mid k < m \}.
\]
This function records which block a generator $u_{ij}$ occupies. For instance,
$u_{ij}$ and $u_{\ell m}$ belong to the same block if and only if $r(i) =
r(\ell)$ and $r(j) = r(m)$.

The simple positive roots of $\mathfrak{sl}(n)$ ($= \operatorname{Lie}(SL_n)$)
and simple reflections in the Weyl group $W$ will be denoted by $\alpha_i$,
$s_i$, ($i\in \{1,\dots, n - 1\}$), respectively. Fix a subset
\[
    J \subseteq \{1,\dots, n - 1\},
\]
and let $W^J \subseteq W$ be the subgroup of $W$ generated by the simple
reflections $\{s_{i} | i \notin J \}$.  Let $w_0$ and $w_0^J$ denote the
longest elements in $W$ and $W^J$ respectively. The corresponding parabolic
element of the Weyl group will be denoted by
\[ w_J :=
	w_0^Jw_0 \in W.
\]
The element $w_J$ can be characterized as the unique element of the Weyl group
such that the set of roots $\Delta_{w_J} := \Delta_+ \cap w_J(\Delta_-)$
coincides with the set of positive roots $\{\beta \in \Delta_+ \mid \beta \geq
\alpha_j \text{ for some } j\in J\}$.

The specific form for the commutation relation between a pair of generators,
say $u_{ij}$ and $u_{\ell m}$, in depends on the relative ordering on $i$, $j$,
$\ell$, $m$, $r(j)$, $r(\ell)$, and $r(m)$, and in some cases on the relative
ordering of $w_J^{-1}(i)$, $w_J^{-1}(\ell)$, $w_0^J(i)$, and $w_0^J(\ell)$,
where we have tacitly identified the Weyl group of $\mathfrak{sl}(n)$ with the
symmetric group on $\{1,\dots, n\}$.

\begin{theorem} \cite[Theorems 3.35 and 3.50]{Jaramillo 2}
    \label{intro, defining relns, algebra of coinvariants}

    The algebra ${\mathcal O}_q(P_J)^{\operatorname{co}\theta}$ is generated by
    $u_{ij}$ (for $(i, j) \in \Phi_J$) and has the following defining
    relations:
    \begin{align}
        u_{ij}u_{\ell m} &=
            \begin{cases}
                qu_{\ell m}u_{ij}
                &
                $\begin{tabularx}{\textwidth}{l}
                    \hspace{5mm}
                    $(\ell = i < j < m)$
                    \\
                    \hspace{5mm}
                    or $(j = m \text{ and } w_J^{-1}(\ell) < w_J^{-1}(i))$
                \end{tabularx}$
                \\
                \\
                u_{\ell m}u_{ij}
                &
                $\begin{tabularx}{\textwidth}{l}
                    \hspace{5mm}
                    $(w_0^J(\ell) < w_0^J(i) < j < m)$
                    \\
                    \hspace{5mm}
                    or
                    $(i < j < \ell < m)$
                    \\
                    \hspace{5mm}
                    or
                    $(i \leq r(j) < \ell < j < m)$
                \end{tabularx}$
                \\
                \\
                u_{\ell m}u_{ij} + \widehat{q}u_{\ell j}u_{im}
                &
                $\begin{tabularx}{\textwidth}{l}
                    \hspace{5mm}
                    $(r(i) < \ell < i < j < m)$
                    \\
                    \hspace{5mm}
                    or
                    $(i \leq r(\ell) < \ell \leq r(j) < j < m)$
                \end{tabularx}$
                \\
                \\
                q^{-1}u_{\ell m}u_{ij} - \widehat{q}u_{(im), \ell}
                &
                $\begin{tabularx}{\textwidth}{l}
                    \hspace{5mm}
                    $(i < j = \ell < m)$
                \end{tabularx}$
            \end{cases}
        \end{align}
        \noindent where $u_{(im), \ell} := (-q)^{r(\ell) -
        w_0^J(\ell)} u_{im} + \sum_{r(\ell) < k < w_0^J(\ell)}(-q)^{\ell -
        w_0^J(k)} u_{w_0^J(k),m} u_{i,w_0^J(k)}$.

\end{theorem}

We prove that ${\mathcal O}_q(P_J)^{\operatorname{co} \theta}$ is isomorphic to
a quantized nilradical of a parabolic subalgebra of $\mathfrak{sl}(n)$ (Theorem
\ref{intro, main theorem}).  In proving this result we first construct
presentations for the quantized nilradicals.  Then we compare these
presentations with the presentations of ${\mathcal O}_q(P_J)^{\operatorname{co}
\theta}$.

Quantized nilradicals belong to a larger family of algebras called quantum
Schubert cell algebras, which were introduced by De Concini, Kac, and Procesi
\cite{DKP} and Lusztig \cite{L}. Quantum Schubert cells play important roles in
ring theory \cite{MC, Y}, crystal/canonical basis theory \cite{Kashiwara,
Lusztig}, and cluster algebras \cite{GLS, Goodearl Yakimov}.  For a complex
semisimple Lie algebra $\mathfrak{g}$ with a root system $\Delta = \Delta_-
\sqcup \Delta_+$, triangular decomposition $\mathfrak{g} = \mathfrak{n}_-
\oplus \mathfrak{h} \oplus \mathfrak{n}_+$, and $w$ an element in the Weyl
group $W_{\mathfrak{g}}$, the corresponding quantum Schubert cell algebras
${\mathcal U}_q^\pm[w]$ are quantizations of the universal enveloping algebra
${\mathcal U}(\mathfrak{n}_{\pm} \cap
\operatorname{ad}(w)(\mathfrak{n}_{\mp}))$. The standard presentation of
${\mathcal U}_q^{\pm}[w]$ typically involves a generating set of variables
$\{X_{\beta}\}$ indexed by roots $\beta$ in $\Delta_w := \Delta_{\pm} \cap
w.(\Delta_{\mp})$.  With respect to a convex order on the roots in $\Delta_w$,
ordered monomials form a basis of ${\mathcal U}_q^{\pm}[w]$.

The quantum Schubert cell algebras of interest in this paper are those of the
form ${\mathcal U}_q^+[w_J]$, where $w_J$ is a parabolic element in the Weyl
group and $\mathfrak{g} = \mathfrak{sl}(n)$. Algebras of this type are
quantizations of nilradicals $\mathfrak{n}_J$ of parabolic subalgebras
$\mathfrak{p}_J \subseteq \mathfrak{sl}(n)$. We will refer to these particular
quantum Schubert cell algebras as \textit{quantized nilradicals} and denote
them by ${\mathcal U}_q(\mathfrak{n}_J)$.

In Theorem \ref{defining relns in q-Schubert cell}, we give a presentation of
the quantized nilradical ${\mathcal U}_q(\mathfrak{n}_J)$. In the extremal
case, when $J$ is the empty set, we have ${\mathcal
U}_q(\mathfrak{n}_{\emptyset}) \cong \mathbb{K}$. At the other extreme,
${\mathcal U}_q(\mathfrak{n}_{\{1,\dots, n - 1\}}) \cong {\mathcal
U}_q(\mathfrak{n}_+)$.  When $J$ is a singleton, say $J = \{p\}$, we have an
isomorphism ${\mathcal U}_q(\mathfrak{n}_{\{p\}}) \cong {\mathcal O}_q(M_{p, n
- p})$, where ${\mathcal O}_q(M_{p, n - p})$ is the algebra of quantum $p
\times (n - p)$ matrices.  In any case, the roots in $\Delta_{w_J}$, as well as
the generators of the algebra ${\mathcal U}_q(\mathfrak{n}_J)$, can be indexed
by $\Phi_J$. We will denote the generators of ${\mathcal U}_q(\mathfrak{n}_J)$
by
\[
    X_{ij}, (i, j) \in \Phi_J.
\]

In Section \ref{proof of defining relns, section}, we prove the following
theorem.

\begin{theorem}
    \label{intro, defining relns of QSC}

    The quantized nilradical ${\mathcal U}_q(\mathfrak{n}_J)$ is generated by
    the root vectors $X_{ij}$ $((i, j) \in \Phi_J)$ and has the following
    defining relations:
    \begin{align}
        X_{ij}X_{\ell m} &=
            \begin{cases}
                qX_{\ell m}X_{ij}
                &
                $\begin{tabularx}{\textwidth}{l}
                    $(\ell < i \text{ and } j = m)$
                    \\
                    or
                    $(\ell = i \text{ and } w_J^{-1}(j) <
                    w_J^{-1}(m))$
                \end{tabularx}$
                \\
                \\
                X_{\ell m}X_{ij}
                &
                $\begin{tabularx}{\textwidth}{l}
                    $(\ell < i < w_0^J(j) < w_0^J(m))$
                    \\
                    or
                    $(\ell < m < i < j)$
                    \\
                    or
                    $(\ell \leq r(m) < i < m < j)$
                \end{tabularx}$
                \\
                \\
                X_{\ell m}X_{ij} + \widehat{q}X_{\ell j}X_{im}
                &
                $\begin{tabularx}{\textwidth}{l}
                    $(\ell < i \leq r(m) < j < m)$
                    \\
                    or
                    $(\ell < i \leq r(m) < m \leq r(j) < j)$
                \end{tabularx}$
                \\
                \\
                q^{-1}X_{\ell m}X_{ij} + X_{(\ell j), m}
                &
                $\begin{tabularx}{\textwidth}{l}
                    $(\ell < m = i < j)$
                \end{tabularx}$
            \end{cases}
        \end{align}
        \noindent where $X_{(\ell j), m} := (-q)^{m - r(m) - 1} X_{\ell j} +
        \widehat{q}\sum_{r(m) < k < m}(-q)^{m - k - 1} X_{kj} X_{\ell k}$.

\end{theorem}

For a subset $J \subseteq \{1,\dots, n - 1\}$, define
\[
    \widetilde{J} := \{ n - j \mid j \in J \} \subseteq \{1,\dots, n - 1\}.
\]
There is a one-to-one correspondence between the generators of ${\mathcal
U}_q(\mathfrak{n}_J)$ and ${\mathcal U}_q(\mathfrak{n}_{\widetilde{J}})$ given
by ``reflecting about the anti-diagonal'': $X_{ij} \longleftrightarrow
X_{w_0(j), w_0(i))}$. Using the defining relations of ${\mathcal
U}_q(\mathfrak{n}_J)$, we can easily verify that there are algebra isomorphisms
\begin{equation}
    \label{isomorphisms of QSCs}
    {\mathcal U}_q(\mathfrak{n}_J) \stackrel{\cong}{\longrightarrow} {\mathcal
    U}_q(\mathfrak{n}_{\widetilde{J}})^{\operatorname{op}},
    \hspace{10mm}
    {\mathcal U}_q(\mathfrak{n}_J) \stackrel{\cong}{\longrightarrow} {\mathcal
    U}_{q^{-1}}(\mathfrak{n}_J)^{\operatorname{op}},
\end{equation}
\noindent given by $X_{ij} \mapsto X_{w_{\widetilde{J}}(j),
w_{\widetilde{J}}(i)}$ and $X_{ij} \mapsto -qX_{ij}$, respectively, for all
$(i, j) \in \Phi_J$.

Since we have presentations for ${\mathcal U}_q(\mathfrak{n}_J)$ and ${\mathcal
O}_q(P_J)^{\operatorname{co}\theta}$ (Theorems \ref{intro, defining relns,
algebra of coinvariants} and \ref{intro, defining relns of QSC}), it is routine
to verify that there is an algebra isomorphism
\begin{equation}
    \label{qsc --> Oq(PJ)}
    {\mathcal U}_q(\mathfrak{n}_J) \stackrel{\cong}{\longrightarrow} {\mathcal
    O}_{q^{-1}}(P_J)^{\operatorname{co}\theta}
\end{equation}
\noindent defined by
\begin{align*}
X_{ij} \mapsto \frac{(-1)^{w_0^J(j) -
w_0^{J}(i)}}{\widehat{q}} u_{w_0^J(i), w_0^J(j)}
\end{align*}
for all $(i,j) \in
\Phi_J$.  Composing the isomorphisms of \ref{isomorphisms of QSCs} with the
isomorphism \ref{qsc --> Oq(PJ)} above gives us the following theorem.

\begin{theorem}
    \label{intro, main theorem}

    There is an algebra isomorphism
    \[
        \Psi: {\mathcal U}_q(\mathfrak{n}_{\widetilde{J}}) \to {\mathcal
        O}_q(P_J)^{\operatorname{co} \theta}
    \]
    \noindent given by $X_{ij} \mapsto \frac{q (-1)^{i + j}}{\widehat{q}}
    u_{w_0(j), w_0(i)}$ for all $(i, j) \in \Phi_{\widetilde{J}}$.

\end{theorem}

\noindent In the extremal case, when $J = \{1,\dots, n - 1\}$, this isomorphism
appears in \cite[Theorem 17]{Jaramillo 1}.

\section{Quantized nilradicals of parabolic subalgebras of $\mathfrak{sl}(n)$}

\subsection{The quantum enveloping algebra ${\mathcal U}_q(\mathfrak{g})$}

Let $\mathfrak{g}$ be a complex semisimple Lie algebra $\mathfrak{g}$ of rank
$\ell$, and let $\Pi = \{\alpha_1,\dots, \alpha_\ell\}$ be a base of simple
roots with respect to a Cartan subalgebra $\mathfrak{h} \subseteq
\mathfrak{g}$.  We will denote the root lattice by $Q = \mathbb{Z}\Pi$. Here
and below, for $p \in \mathbb{N}$, we put $[p]:=\{1,\dots, p\}$. We will denote
the Chevalley generators of the quantum universal enveloping algebra ${\mathcal
U}_q(\mathfrak{g})$ by
\[
    E_i, F_i, (i \in [\ell]), \text{ and } K_\mu (\mu\in Q).
\]
\noindent The algebra ${\mathcal U}_q(\mathfrak{g})$ has a triangular
decomposition
\[
    {\mathcal U}_q(\mathfrak{g}) \cong {\mathcal U}_q^-(\mathfrak{g}) \otimes
    {\mathcal U}_q^0(\mathfrak{g}) \otimes {\mathcal U}_q^+(\mathfrak{g}),
\]
\noindent where ${\mathcal U}_q^+(\mathfrak{g})$ is the subalgebra generated by
the $E_i$'s, ${\mathcal U}_q^-(\mathfrak{g})$ is the subalgebra generated by
the $F_i$'s, and ${\mathcal U}_q^0(\mathfrak{g})$ is the subalgebra generated
by the $K_\mu$'s (see e.g. \cite{CP, Jantzen, KS}).

While quantum enveloping algebras can be associated to semisimple Lie algebras,
or more generally Kac-Moody Lie algebras, we focus on the case when
$\mathfrak{g}$ is the special linear Lie algebra $\mathfrak{sl}(n)$. We use the
standard realization of $\mathfrak{sl}(n)$ as the Lie algebra of traceless $n
\times n$ matrices.  Let $\mathfrak{h}$ be the Cartan subalgebra of diagonal
matrices in $\mathfrak{sl}(n)$. The simple roots are $\alpha_i = e_i - e_{i +
1}\in \mathfrak{h}^*$ ($i \in [n - 1]$), where $e_\ell$ ($\ell \in [n]$) is the
linear functional on $\mathfrak{h}$ that returns the $\ell$-th entry along the
diagonal.  Let $\langle, \rangle$ be the symmetric bilinear form on
$\mathfrak{h}^*$ defined by the rule $\langle e_i,e_j\rangle = \delta_{ij}$ for
all $i,j\in [n]$.  The quantum enveloping algebra ${\mathcal
U}_q(\mathfrak{sl}(n))$ is the associative $\mathbb{K}$-algebra generated by
$F_i, E_i, K_\mu$ ($i\in [n - 1]$, $\mu\in Q$), and has the defining relations
\begin{align}
    &K_0 = 1,
    & &K_\mu K_\rho=K_{\mu+\rho},
    \\
    &K_\mu E_i = q^{\left< \mu ,\alpha_i\right>} E_iK_\mu,
    & &K_\mu F_i = q^{-\left< \mu ,\alpha_i\right>} F_iK_\mu,
    \\
    &E_iF_j =
    F_jE_i+\delta_{ij}\frac{K_{\alpha_i}-K_{-\alpha_i}}{\widehat{q}},
    & &
\end{align}
\noindent for all $i,j\in [n - 1]$ and $\mu, \rho \in Q$, together with the
$q$-Serre relations
\begin{align}
    &\label{q-Serre 2}E_i^2E_j - (q+ q^{-1})E_iE_jE_i + E_jE_i^2 = 0, &(|i - j|
    = 1),\\
    &F_i^2F_j - (q+ q^{-1})F_iF_jF_i + F_jF_i^2 = 0, &(|i - j| = 1),\\
    &\label{q-Serre 1}
    E_iE_j = E_jE_i, \hspace{1cm} F_iF_j = F_jF_i, &(|i - j | > 1).
\end{align}

\subsection{The symmetries $\mathbf{T_i}$}

The Weyl group and braid group of $\mathfrak{g}$ will be denoted respectively
by
\[
    W_\mathfrak{g} = \langle s_1, \dots, s_\ell\rangle, \hspace{5mm}  {\mathcal
    B}_\mathfrak{g} = \langle T_1,\dots T_\ell\rangle.
\]
\noindent If $w = s_{i_1}\cdots s_{i_m} \in W_\mathfrak{g}$ is a reduced
expression, we use the abbreviation
\[
    T_w := T_{i_1}\cdots T_{i_m}.
\]
\noindent In \cite[Section 37.1.3]{L}, Lusztig defines an action of the braid
group ${\mathcal B}_{\mathfrak{g}}$ via algebra automorphisms on ${\mathcal
U}_q(\mathfrak{g})$.  In fact, Lusztig defines the symmetries $T_{i,1}^\prime$,
$T_{i,-1}^{\prime}$, $T_{i,1}^{\prime\prime}$, and $T_{i,-1}^{\prime\prime}$
($i \in [\ell]$). By \cite[Proposition 37.1.2]{L}, these are automorphisms of
${\mathcal U}_q(\mathfrak{g})$, while by \cite[Theorem 39.4.3]{L} they satisfy
the braid relations. The following proposition is a key property of the braid
group action (see e.g. \cite[Proposition 8.20]{Jantzen}).

\begin{proposition}

    \label{braid action 1}

    If $w\in W_{\mathfrak{g}}$, $\alpha \in \Pi$, and $w(\alpha) \in \Pi$, then
    $T_w(E_\alpha) = E_{w(\alpha)}$.

\end{proposition}

When $\mathfrak{g}$ is the Lie algebra $\mathfrak{sl}(n)$ and $T_i =
T_{i,1}^{\prime\prime}$ ($i \in [n - 1]$), Lusztig's symmetries can be
succinctly written as
\begin{align}
    &T_i(K_\mu) = K_{s_i(\mu)},
    \\
    & T_i(E_j) =
    \label{Lusztig symmetries}
    \begin{cases}
        E_j, & (| i - j | > 1), \\
        E_iE_j - q^{-1}E_jE_i, & (|i - j| = 1), \\
        -F_iK_{\alpha_i} & (| i - j | = 0),
    \end{cases} \\
    & T_i(F_j) =
    \begin{cases}
        F_j, & (| i - j | > 1), \\
        -q(F_iF_j - q^{-1}F_jF_i), & (| i - j | = 1), \\
        -K_{\alpha_i}^{-1}E_i & (| i - j | = 0).
    \end{cases}
\end{align}
\noindent for all $i,j\in [n - 1]$ and $\mu\in Q$.

\subsection{Quantum Schubert cells}
\label{q-Schubert cell, section}

Quantum Schubert cell algebras were introduced in \cite{DKP} and \cite{L}. They
are a family of subalgebras of ${\mathcal U}_q(\mathfrak{g})$ indexed by
elements of the Weyl group of $\mathfrak{g}$.  To construct a quantum Schubert
cell, first fix $w\in W_\mathfrak{g}$ and a reduced expression $w =
s_{i_1}\cdots s_{i_t} \in W_\mathfrak{g}$. Next define the positive roots
\begin{equation}
    \beta_1 = \alpha_{i_1}, \beta_2=s_{i_1}\alpha_{i_2}, \dots, \beta_t =
    s_{i_1} \cdots s_{i_{t-1}}\alpha_{i_t}
\end{equation}
and the positive root vectors
\begin{equation}
    X_{\beta_1}=E_{i_1},
    X_{\beta_2}=T_{s_{i_1}}E_{i_2},...,X_{\beta_t}=T_{s_{i_1}}\cdots
    T_{s_{i_{t-1}}}E_{i_t}.
\end{equation}
\noindent There is an analogous construction of negative root vectors
$X_{-\beta_1}, \dots, X_{-\beta_t}$ by replacing the $E_i$'s with $F_i$'s in
the above construction. Following \cite{DKP} and \cite[Section 40.2]{L}, the
quantum Schubert cell algebra ${\mathcal U}_q^{\pm}[w]$ is defined to be the
subalgebra of ${\mathcal U}_q(\mathfrak{g})$ generated by the root vectors
$X_{\pm\beta_1}, \dots, X_{\pm\beta_t}$. De Concini, Kac, and Procesi
\cite[Proposition 2.2]{DKP} and Lusztig \cite{L} proved that ${\mathcal
U}_q^\pm[w]$ does not depend on the reduced expression for $w$.  It was
conjectured by Berenstein and Greenstein in \cite[Conjecture 5.3]{AB} that
${\mathcal U}_q^{\pm}[w]$ could equivalently be defined as
\[
    {\mathcal U}_q^\pm[w] = {\mathcal U}_q^\pm(\mathfrak{g}) \cap T_w({\mathcal
    U}_q^\mp(\mathfrak{g}))
\]
\noindent for any symmetrizable Kac-Moody Lie algebra $\mathfrak{g}$. They
proved their conjecture in the case when $\mathfrak{g}$ is of finite type. The
conjecture was later proven independently by Kimura \cite[Theorem 1.1
(1)]{Kimura} and Tanisaki \cite[Proposition 2.10]{Tanisaki}.

The algebra ${\mathcal U}_q^\pm[w]$ has a PBW-type basis of ordered monomials
\begin{equation}
    \label{PBW basis}
    X_{\pm\beta_1}^{n_1}\cdots X_{\pm\beta_t}^{n_t},\hspace{.4cm}
    n_1,...,n_t\in\mathbb{Z}_{\geq 0}.
\end{equation}
\noindent The Levendorskii-Soibelmann Straightening Rule gives commutation
relations in ${\mathcal U}_q^+[w]$.

\begin{theorem}\cite[Prop. 5.5.2]{LS}

    For $i<j$,
    \begin{equation}
        X_{\beta_i}X_{\beta_j} = q^{\langle\beta_i, \beta_j\rangle} X_{\beta_j}
        X_{\beta_i} + \sum_{n_{i+1}, ..., n_{j - 1} \geq 0} z_{ij}(n_{i + 1},
        ..., n_{j - 1}) X_{\beta_{i+1}}^{n_{i+1}} \cdots X_{\beta_{j -
        1}}^{n_{j - 1}},
    \end{equation}
    \noindent where $z_{ij}(n_{i+1},...,n_{j-1})\in\mathbb{K}$, and
    $z_{ij}(n_{i+1},\dots, n_{j-1}) = 0$ whenever $\sum_{i<k<j}n_k\beta_k \neq
    \beta_i + \beta_j$.

\end{theorem}

\noindent An analogous straightening rule applies to ${\mathcal U}_q^-[w]$. The
straightening law in conjunction with the PBW basis result \ref{PBW basis} can
be used to give finite presentations of quantum Schubert cell algebras.

\subsection{The quantum Schubert cell algebras $\mathbf{{\mathcal
U}_q(\mathfrak{n}_J)}$}

To construct the quantum Schubert cell algebras of interest in the remainder of
this paper, we turn our attention to parabolic elements $w_J$ in the Weyl group
$W$ of $\mathfrak{sl}(n)$. The algebras ${\mathcal U}_q^+[w_J]$ are
quantizations of the nilradicals $\mathfrak{n}_J$ of standard parabolic
subalgebras $\mathfrak{p}_J$ of $\mathfrak{sl}(n)$. For this reason we denote
them instead by ${\mathcal U}_q(\mathfrak{n}_J)$. First let
\[
    J = \{i_1 < i_2 \dots < i_t\} \subseteq [n - 1],
\]
\noindent and let $W^J \subseteq W$ be the subgroup of $W$ generated by the
simple reflections $\{s_{i} | i \notin J \}$.  Let $w_0$ and $w_0^J$ denote the
longest elements in $W$ and $W^J$ respectively. Define
\[
    w_J := w_0^Jw_0 \in W.
\]
\noindent We can write $w_J = S_tS_{t-1}\dots S_1$, where
\begin{equation}
    \label{S_k's} S_k := (s_{i_k}\dots s_{n-1})(s_{i_k - 1} \dots s_{n-2})
    \dots (s_{i_{k-1} + 1} \dots s_{n - (i_k - i_{k-1})}) \in W
\end{equation}
\noindent for all $k \in [t]$ (where, by convention, $i_0 = 0$ and $i_{t+1} =
n$). The expression written in \ref{S_k's} is a reduced expression, and one can
obtain a reduced expression for $w_J$ by concatenating these reduced
expressions for $S_t$, $S_{t-1}$, \dots, $S_1$. We tacitly use this particular
reduced expression for $w_J$ in our construction of ${\mathcal
U}_q(\mathfrak{n}_J)$.  The set of positive Lusztig roots is
\[
    \Delta_{w_J} = \{ e_i - e_j \in Q \mid \exists k\in J \text{ such that } i
    \leq k < j \}.
\]
\noindent We find it convenient to define the set of tuples
\[
    \Phi_J := \{ (i, j) \in [n] \times [n] \mid e_i - e_j \in
    \Delta_{w_J}\}.
\]
\noindent For brevity we let $X_{ij}$ ($(i, j) \in \Phi_J$) denote the
positive Lusztig root vector of degree $e_i-e_j$.

\begin{example}

    Let $n = 7$ and suppose $J = \{ 2, 5, 6\}$. Here,
    \[
        w_J = s_6(s_5s_6)(s_4s_5)(s_3s_4)(s_2s_3s_4s_5s_6)(s_1s_2s_3s_4s_5) \in
        W
    \]
    \noindent is a reduced expression. The positive Lusztig root vectors in the
    corresponding quantized nilradical ${\mathcal U}_q(\mathfrak{n}_J)$ are
    $X_{67}$, $X_{57}$, $X_{56}$, $X_{47}$, $X_{46}$, $X_{37}$, $X_{36}$,
    $X_{27}$, $X_{26}$, $X_{23}$, $X_{24}$, $X_{25}$, $X_{17}$, $X_{16}$,
    $X_{13}$, $X_{14}$, $X_{15}$.

\end{example}

\noindent If the root vector $X_{ij}$ is viewed as occupying the $(i,j)$-entry
in an $n\times n$ array, the entire set $\{X_{ij}\}$ of root vectors of
${\mathcal U}_q(\mathfrak{n}_J)$ fills all entries in a block upper triangular
shape that depends on $J$. More precisely, the set of positive Lusztig roots
$\Delta_{w_J}$ can be characterized as the smallest set satisfying the
conditions (1) $e_j - e_{j + 1} \in \Delta_{w_J}$ if and only if $j \in J$, (2)
if $e_i - e_j \in \Delta_{w_J}$ with $i > 1$, then $e_{i - 1} - e_j \in
\Delta_{w_J}$, and (3) if $e_i - e_j \in \Delta_{w_J}$ with $j < n$, then $e_i
- e_{j + 1} \in \Delta_{w_J}$.

The Levendorskii-Soibelmann straightening rule together with the PBW basis
result \ref{PBW basis} implies that a finite presentation of a quantum Schubert
cell algebra can be obtained from the commutation relations among the pairs of
root vectors. In order to describe the defining relations of ${\mathcal
U}_q(\mathfrak{n}_J)$, we first define the function $r: [n] \to J \cup \{0\}$
as \[ r(m):= \max\{ k \in J \cup \{0\} \mid k < m \}.  \] The specific form for
the commutation relation between a pair of root vectors, say $X_{ij}$ and
$X_{\ell m}$, depends on the relative ordering on $i$, $j$, $\ell$, $m$,
$r(j)$, and $r(m)$, and in some cases the relative ordering on $w_J^{-1}(j)$,
$w_J^{-1}(m)$, $w_0^J(j)$, and $w_0^J(m)$ plays a role, where we have
identified the Weyl group of $\mathfrak{sl}(n)$ with the symmetric group on
$[n]$. Under this identification the simple reflection $s_i$ corresponds to the
transposition $(i, i + 1)$, and $w_J$ corresponds to the unique permutation
satisfying the condition that for all $1 \leq i < j \leq n$, $w_J(i) < w_J(j)$
if and only if $r(i) = r(j)$.

\begin{theorem}
    \label{defining relns in q-Schubert cell}

    The quantized nilradical ${\mathcal U}_q(\mathfrak{n}_J)$ is generated by
    the root vectors $X_{ij}$ $((i, j) \in \Phi_J)$ and has the following
    defining relations:
    \begin{align}
        X_{ij}X_{\ell m} &=
            \begin{cases}
                qX_{\ell m}X_{ij}
                &
                $\begin{tabularx}{\textwidth}{l}
                    $(\ell < i \text{ and } j = m)$
                    \\
                    or
                    $(\ell = i \text{ and } w_J^{-1}(j) <
                    w_J^{-1}(m))$
                \end{tabularx}$
                \\
                \\
                X_{\ell m}X_{ij}
                &
                $\begin{tabularx}{\textwidth}{l}
                    $(\ell < i < w_0^J(j) < w_0^J(m))$
                    \\
                    or
                    $(\ell < m < i < j)$
                    \\
                    or
                    $(\ell \leq r(m) < i < m < j)$
                \end{tabularx}$
                \\
                \\
                X_{\ell m}X_{ij} + \widehat{q}X_{\ell j}X_{im}
                &
                $\begin{tabularx}{\textwidth}{l}
                    $(\ell < i \leq r(m) < j < m)$
                    \\
                    or
                    $(\ell < i \leq r(m) < m \leq r(j) < j)$
                \end{tabularx}$
                \\
                \\
                q^{-1}X_{\ell m}X_{ij} + X_{(\ell j), m}
                &
                $\begin{tabularx}{\textwidth}{l}
                    $(\ell < m = i < j)$
                \end{tabularx}$
            \end{cases}
        \end{align}
        \noindent where $X_{(\ell j), m} := (-q)^{m - r(m) - 1} X_{\ell j} +
        \widehat{q}\sum_{r(m) < k < m}(-q)^{m - k - 1} X_{kj} X_{\ell k}$.

\end{theorem}

\begin{proof}

    See Section \ref{proof of defining relns, section}.

\end{proof}

\section{The Algebra ${\mathcal O}_q(P_J)^{\operatorname{co}\theta}$ of
coinvariants}

\subsection{Crossed product algebras, $H$-cleft extensions, and
coinvariants}

Let $H$ be a bialgebra and let $A$ be a left $H$-comodule algebra with coaction
$\xi: A \to H \otimes A$. An element $a\in A$ is a \textit{left coinvariant} if
$\xi(a) = 1 \otimes a$. The set $A^{\operatorname{co}\xi}$ of left coinvariants
is, in fact, a subalgebra of $A$. The $H$-extension $A^{\operatorname{co} \xi}
\subseteq A$ is called \textit{$H$-cleft} if there exists a convolution
invertible morphism of left $H$-comodules $\gamma: H \to A$ (with convolution
inverse $\overline{\gamma}$) such that $\gamma(1) = 1$. In this setting, the
vector space $H\otimes A^{\operatorname{co}\xi}$ can be equipped with an
associative multiplication giving it the structure of a left crossed product
algebra $H \#_\sigma^\triangleleft A^{\operatorname{co}\xi}$ \cite{Doi
Takeuchi, Montgomery}. The multiplication in $H \#_\sigma^\triangleleft
A^{\operatorname{co}\xi}$ is constructed by using the cleavage map $\gamma$ to
first define a right $H$-action on $A^{\operatorname{co}\xi}$ given by
\[
    a \triangleleft h := \sum_{(h)} \overline{\gamma}(h_1)a\gamma(h_2)
\]
\noindent and a linear map $\sigma: H \otimes H \to A^{\operatorname{co}\xi}$
defined as
\[
    \sigma(h, h^\prime) = \sum_{(h), (h^\prime)} \overline{\gamma} (h_1
    h_1^\prime) \gamma(h_2) \gamma(h_2^\prime),
\]
\noindent for all $a, a^\prime \in A^{\operatorname{co}\xi}$, $h, h^\prime \in
H$.  The multiplication in $H \#_\sigma^\triangleleft A^{\operatorname{co}\xi}$
is defined by
\begin{equation}
    \label{multiplication in a left crossed product}
    (h \otimes a)(h^\prime \otimes a^\prime) = \sum_{(h),
    (h^\prime)}h_1h_1^\prime \otimes \sigma(h_2,
    h_2^\prime)(a \triangleleft h_3^\prime)a^\prime,
\end{equation}
\noindent for all $a, a^\prime \in A^{\operatorname{co}\xi}$, $h, h^\prime \in
H$. In the case when $\gamma: H \to A$ is an algebra morphism, the
multiplication in $H \#_\sigma^\triangleleft A^{\operatorname{co}\xi}$
simplifies to
\begin{equation}
    \label{multiplication in smash product}
    (h \otimes a)(h^\prime \otimes a^\prime) = \sum_{(h)} hh_1^\prime \otimes
    (a \triangleleft h_2^\prime)a^\prime
\end{equation}
\noindent for all $h,h^\prime \in H$ and $a, a^\prime \in
A^{\operatorname{co}\theta}$, which is precisely the multiplication in the left
smash product algebra, commonly denoted $H \# A^{\operatorname{co}\xi}$.  By a
result of \cite{Doi Takeuchi}, shown in \cite[Proposition 7.2.3]{Montgomery},
there is an algebra isomorphism
\[
    \Psi: H\#_\sigma^\triangleleft A^{\operatorname{co}\xi}
    \stackrel{\cong}{\longrightarrow} A
\]
\noindent given by $h \otimes a \mapsto \gamma(h)a$ for all $h \in H$ and $a
\in A^{\operatorname{co}\xi}$.

\subsection{The algebra ${\mathcal O}_q(P_J)^{\operatorname{co}\theta}$}
\label{algebra of coinvariants, section}

Recall we fix a subset
\[
    J \subseteq [n - 1].
\]
\noindent Let $P_J$ be the standard parabolic subgroup of $SL_n$ obtained by
deleting the negative simple roots $-\alpha_i$ for $i \in J$. Thus, $P_J$ is a
group of block upper triangular matrices,
\[
    P_J = \{ (a_{ij})\in SL_n \mid a_{ij} = 0 \text{ if } (j, i) \in
    \Phi_J\}.
\]
\noindent We will denote the Levi decomposition of $P_J$ by
\[
    P_J = L_JU_J,
\]
\noindent where $L_J$ is the standard Levi factor of $P_J$ consisting of
matrices with block entries off the main diagonal equal to $0$, and $U_J$ is
the unipotent subgroup of matrices in $P_J$ having block diagonal entries equal
to identity matrices. Observe that matrix multiplication
\[
    L_J \times P_J \to P_J,
\]
\noindent induces a coaction ${\mathcal O}(P_J) \to {\mathcal O}(L_J) \otimes
{\mathcal O}(P_J)$ among coordinate rings. Following \cite{Jaramillo 1,
Jaramillo 2}, we turn our attention to the quantum analogue of this coaction
\[
    \theta: {\mathcal O}_q(P_J) \to {\mathcal O}_q(L_J) \otimes {\mathcal
    O}_q(P_J),
\]
\noindent where ${\mathcal O}_q(L_J)$ and ${\mathcal O}_q(P_J)$ are the
quantized coordinate rings of $L_J$ and $P_J$ respectively.  We recall
${\mathcal O}_q(L_J)$ and ${\mathcal O}_q(P_J)$ are obtained from ${\mathcal
O}_q(SL_n)$ by quotienting by certain two-sided ideals. First, the quantized
coordinate ring ${\mathcal O}_q(SL_n)$ is the $\mathbb{K}$-algebra generated by
$x_{ij}$ ($i,j \in [n])$ and has defining relations
    \begin{align*}
        x_{ij}x_{\ell m}
        &=
        \begin{cases}
            qx_{\ell m}x_{ij},
            &
            (i < \ell \text{ and } j = m) \text{ or } (i = \ell \text{ and } j
            < m),
            \\
            x_{\ell m}x_{ij},
            &
            (i < \ell \text{ and } m < j),
            \\
            x_{\ell m}x_{ij} + \widehat{q}x_{\ell j}x_{im},
            &
            (i < \ell \text{ and } j < m),
        \end{cases}
    \end{align*}
\noindent together with the relation that sets the quantum determinant equal to $1$,
\begin{equation}
    \label{q-determinant is 1}
    \operatorname{det}_q := \sum_{\sigma \in \operatorname{Sym}(n)}
    (-q)^{\ell(\sigma)}x_{1, \sigma(1)}\cdots x_{n, \sigma(n)} = 1,
\end{equation}
\noindent where $\ell(\sigma)$ is the number of inversions in $\sigma$ (see
e.g. \cite{KS}).  Furthermore ${\mathcal O}_q(SL_n)$ is a Hopf algebra
\cite[Section 9.2.3, Proposition 10]{KS} with comultiplication $\Delta$, counit
$\epsilon$, and antipode $S$, given by
\begin{align*}
    \Delta(x_{ij}) &:= \sum_k x_{ik} \otimes x_{kj}, \\
    \epsilon(x_{ij}) &:= \delta_{ij}, \\
    S(x_{ij}) &:= (-q)^{i - j}[\{1,\dots, \widehat{j},\dots, n\} | \{1,\dots
    \widehat{i},\dots, n\}],
\end{align*}
\noindent where, for a pair of subsets $A = \{a_1 < \dots < a_m\}$ and $B =
\{b_1 < \dots < b_m\}$ of $[n]$ of the same cardinality, the \textit{quantum
minor} $[A | B]$ with \textit{row set} $A$ and \textit{column set} $B$ is
defined as
\[
    [A | B] := \sum_{\sigma\in \operatorname{Sym}(m)}
    (-q)^{\ell(\sigma)}x_{a_1, b_{\sigma(1)}}\cdots x_{a_m,b_{\sigma (m)}} \in
    {\mathcal O}_q(SL_n).
\]
\noindent The quantized coordinate rings ${\mathcal O}_q(P_J)$ and ${\mathcal
O}_q(L_J)$ are obtained from ${\mathcal O}_q(SL_n)$ by quotienting by the
two-sided ideals generated by $\{x_{ij} \mid (j, i) \in \Phi_J\}$ and
$\{x_{ij} \mid (i, j) \in \Phi_J \text{ or } (j, i) \in \Phi_J\}$
respectively,
\begin{align*}
    {\mathcal O}_q(P_J) &:= {\mathcal O}_q(SL_n) / \langle x_{ij} \mid (j, i)
    \in \Phi_J\rangle,
    \\
    {\mathcal O}_q(L_J) &:= {\mathcal O}_q(SL_n) / \langle x_{ij} \mid (i, j)
    \in \Phi_J \text{ or } (j, i) \in \Phi_J\rangle.
\end{align*}
\noindent With a slight abuse of notation, we use the symbol $x_{ij}$ to refer
to the coset in ${\mathcal O}_q(P_J)$ containing $x_{ij}$, whereas we will
adopt the symbol $y_{ij}$ to refer to the coset in ${\mathcal O}_q(L_J)$
containing $x_{ij}$. Hence,
\[
    {\mathcal G}_J := \{y_{ij} \mid (i, j) \in [n] \times [n], (i, j) \not\in
    \Phi_J, (j, i) \not\in \Phi_J \}
\]
\noindent is a set of generators for ${\mathcal O}_q(L_J)$.  The quantized
coordinate ring ${\mathcal O}_q(L_J)$ inherits a Hopf algebra structure from
${\mathcal O}_q(SL_n)$. We will denote the comultiplication, counit, and
antipode of ${\mathcal O}_q(L_J)$ by
\[
    \Delta_L, \epsilon_L, S_L.
\]
\noindent In \cite[Theorem 3.19]{Jaramillo 2} it was shown that ${\mathcal
O}_q(P_J)^{\operatorname{co} \theta} \subseteq {\mathcal O}_q(P_J)$ is a left
$O_q(L_J)$-cleft extension with cleavage map
\[
    \gamma: {\mathcal O}_q(L_J) \to {\mathcal O}_q(P_J)
\]
\noindent being the algebra homomorphism given by $\gamma(y_{ij}) = x_{ij}$ for
all generators $y_{ij} \in {\mathcal G}_J$.  The convolution inverse of
$\gamma$ is $\overline{\gamma} = \gamma \circ S_L$.  Since $\gamma$ is an
algebra homomorphism, the crossed product algebra ${\mathcal
O}_q(L_J)\#_\sigma^\triangleleft {\mathcal O}_q(P_J)^{\operatorname{co}\theta}$
is, in fact, the smash product algebra ${\mathcal O}_q(L_J) \# {\mathcal
O}_q(P_J)^{\operatorname{co}\theta}$ with multiplication given in
\ref{multiplication in smash product}. By the above discussion, we have the
following isomorphism.

\begin{theorem} \cite[Theorem 3.19]{Jaramillo 2}

    There is an algebra isomorphism
    \[
        {\mathcal O}_q(L_J) \# {\mathcal O}_q(P_J)^{\operatorname{co}\theta}
        \stackrel{\cong}{\longrightarrow} {\mathcal O}_q(P_J)
    \]
    \noindent given by $y_{ij} \otimes u \mapsto x_{ij}u$, for all generators
    $y_{ij}\in {\mathcal G}_J$ and $u\in {\mathcal
    O}_q(P_J)^{\operatorname{co}\theta}$.

\end{theorem}

\noindent To give a presentation of ${\mathcal O}_q(P_J)^{\operatorname{co}
\theta}$ in terms of generators and relations, first define the set
\[
    C_i := \{k \in [n] \mid r(k) =r(i)\}.
\]
\noindent for $i \in [n]$. It follows from the quantum determinant relation
\ref{q-determinant is 1} in ${\mathcal O}_q(SL_n)$ that the quantum minor $[C_i
| C_i]$ is invertible in ${\mathcal O}_q(P_J)$ for all $i\in [n]$. For each
$(i, j) \in \Phi_J$, the ratio of quantum minors
\begin{equation}
    \label{ratio of quantum minors}
    u_{ij} := [C_i | C_i]^{-1} [ C_i | C_i\backslash \{i\} \cup \{j\}] \in
    {\mathcal O}_q(P_J)
\end{equation}
\noindent is a left coinvariant \cite[Section 3.3]{Jaramillo 2}. These
particular elements generate ${\mathcal O}_q(P_J)^{\operatorname{co}\theta}$.

\begin{theorem} \cite[Theorems 3.35 and 3.50]{Jaramillo 2}
    \label{defining relns, algebra of coinvariants}

    The algebra ${\mathcal O}_q(P_J)^{\operatorname{co}\theta}$ is generated by
    $u_{ij}$ (for $(i, j) \in \Phi_J$) and has the following defining
    relations:
    \begin{align}
        u_{ij}u_{\ell m} &=
            \begin{cases}
                qu_{\ell m}u_{ij}
                &
                $\begin{tabularx}{\textwidth}{l}
                    \hspace{5mm}
                    $(\ell = i < j < m)$
                    \\
                    \hspace{5mm}
                    or $(j = m \text{ and } w_J^{-1}(\ell) < w_J^{-1}(i))$
                \end{tabularx}$
                \\
                \\
                u_{\ell m}u_{ij}
                &
                $\begin{tabularx}{\textwidth}{l}
                    \hspace{5mm}
                    $(w_0^J(\ell) < w_0^J(i) < j < m)$
                    \\
                    \hspace{5mm}
                    or
                    $(i < j < \ell < m)$
                    \\
                    \hspace{5mm}
                    or
                    $(i \leq r(j) < \ell < j < m)$
                \end{tabularx}$
                \\
                \\
                u_{\ell m}u_{ij} + \widehat{q}u_{\ell j}u_{im}
                &
                $\begin{tabularx}{\textwidth}{l}
                    \hspace{5mm}
                    $(r(i) < \ell < i < j < m)$
                    \\
                    \hspace{5mm}
                    or
                    $(i \leq r(\ell) < \ell \leq r(j) < j < m)$
                \end{tabularx}$
                \\
                \\
                q^{-1}u_{\ell m}u_{ij} - \widehat{q}u_{(im), \ell}
                &
                $\begin{tabularx}{\textwidth}{l}
                    \hspace{5mm}
                    $(i < j = \ell < m)$
                \end{tabularx}$
            \end{cases}
        \end{align}
        \noindent where $u_{(im), \ell} := (-q)^{r(\ell) -
        w_0^J(\ell)} u_{im} + \sum_{r(\ell) < k < w_0^J(\ell)}(-q)^{\ell -
        w_0^J(k)} u_{w_0^J(k),m} u_{i,w_0^J(k)}$.

\end{theorem}

It was also shown in \cite[Theorem 3.46]{Jaramillo 2} that ${\mathcal
O}_q(P_J)^{\operatorname{co}\theta}$ is an iterated skew polynomial ring
\[
    \mathbb{K}[t_1][t_2; \tau_2, \delta_2]\cdots [t_M; \tau_M, \delta_M],
\]
\noindent where $M = \operatorname{dim}(U_J)$ and $t_k$ ($k \in [M]$) is the
$k$-th element in the sequence $\{u_{ij}\}$ of generators ordered via the rule
$u_{ij} \prec u_{\ell m}$ if and only if either (1) $w_0^J(i) < w_0^J(\ell)$,
or (2) $i = \ell$ and $j < m$.

\section{Proof of Theorem \ref{defining relns in q-Schubert cell}: The defining
relations of ${\mathcal U}_q(\mathfrak{n}_J)$}
\label{proof of defining relns, section}

In proving Theorem \ref{defining relns in q-Schubert cell}, which gives the
defining relations in ${\mathcal U}_q(\mathfrak{n}_J)$, we find it convenient
to first introduce the $q^{-1}$-commutator; for $x, y \in {\mathcal
U}_q(\mathfrak{sl}(n))$, define
\[
    [x, y] := xy - q^{-1}yx \in {\mathcal U}_q(\mathfrak{sl}(n)).
\]
\noindent Observe that for all $x, y, z \in {\mathcal U}_q(\mathfrak{sl}(n))$
such that $xz = zx$, we have an associativity property,
\begin{equation}
    \label{q-associativity}
    [[x, y], z] = [x, [y, z]].
\end{equation}
\noindent For $a_1,\dots a_m \in [n - 1]$, we use the abbreviations
\begin{align*}
    &\mathbf{E}_{a_1,\dots, a_m} = [[[ \cdots [E_{a_1}, E_{a_2}], E_{a_3}],
    \cdots ], E_{a_m}] \in {\mathcal U}_q(\mathfrak{sl}(n)),
    \\
    &\mathbf{T}_{a_1,\dots, a_m} = T_{a_1} \circ \cdots \circ T_{a_m} \in
    \operatorname{Aut}({\mathcal U}_q(\mathfrak{sl}(n)).
\end{align*}
\noindent The $q$-Serre relation \ref{q-Serre 1} implies that if a pair of
consecutive indices, say $a_{k}$ and $a_{k + 1}$, differ by more than $1$, then
those indices in the nested $q^{-1}$-commutator $\mathbf{E}_{a_1,\dots,a_m}$
can be interchanged,
\begin{equation}
    \label{q-commutativity}
    \mathbf{E}_{a_1, \dots, a_m} = \mathbf{E}_{a_1,\dots a_{k - 1}, a_{k + 1},
    a_k, a_{k + 2}, \dots, a_m}.
\end{equation}

The following lemma tells us how certain nested $q^{-1}$-commutators behave
under Lusztig's symmetries.

\begin{lemma}

    \label{T_E}

    $ $

    For all $1 \leq k < \ell < n$,

    \begin{enumerate}

        \item \label{T_E ascending, prepend} $T_k(\mathbf{E}_{k + 1, k + 2,
            \dots, \ell}) = \mathbf{E}_{k, k + 1, \dots \ell}$,

        \item \label{T_E ascending, remove} $T_\ell(\mathbf{E}_{k, k + 1,\dots,
            \ell}) = \mathbf{E}_{k,\dots, k + 1,\dots, \ell - 1}$,

        \item \label{T_E descending, prepend} $T_\ell(\mathbf{E}_{\ell - 1,
            \ell - 2, \dots, k}) = \mathbf{E}_{\ell, \ell - 1, \dots, k}$,

        \item \label{T_E descending, remove} $T_k(\mathbf{E}_{\ell, \ell - 1,
            \dots, k}) = \mathbf{E}_{\ell, \ell - 1, \dots, k + 1}$.

    \end{enumerate}

    For all $k, \ell, m \in [n - 1]$ such that $k \leq m$ and $\ell \not\in \{k
    - 1, k , m, m + 1\}$,

    \begin{enumerate}

        \setcounter{enumi}{4}

        \item \label{T_E ascending, invariant} $T_\ell(\mathbf{E}_{k, k +
            1,\dots, m}) = \mathbf{E}_{k, k + 1,\dots, m}$,

        \item \label{T_E descending, invariant} $T_\ell(\mathbf{E}_{m, m -
            1,\dots, k}) = \mathbf{E}_{m, m - 1,\dots, k}$.

    \end{enumerate}

\end{lemma}

\begin{proof}

    Parts \ref{T_E ascending, prepend} and \ref{T_E descending, prepend} follow
    from the definition of the braid group action \ref{Lusztig symmetries}
    together with the fact that the braid group acts via algebra automorphisms.

    To prove part \ref{T_E ascending, remove}, we first consider the case when
    $\ell = k + 1$. In this setting, the desired result follows directly from
    Proposition \ref{braid action 1}, $T_\ell(\mathbf{E}_{\ell - 1, \ell}) =
    T_\ell T_{\ell - 1} (E_\ell) = E_{\ell - 1}$. However, if $\ell > k + 1$,
    we can use \ref{q-associativity} to write $\mathbf{E}_{k, k + 1,\dots,
    \ell} = [\mathbf{E}_{k, k + 1,\dots, \ell - 2}, \mathbf{E}_{\ell - 1,
    \ell}]$. From \ref{Lusztig symmetries}, $E_p$ ($p = k, k + 1,\dots, \ell -
    2$) is fixed by $T_\ell$. Thus
    \begin{align*}
        T_\ell(\mathbf{E}_{k, k + 1,\dots,\ell})
        &=
        T_\ell([\mathbf{E}_{k, k + 1,\dots, \ell - 2}, \mathbf{E}_{\ell - 1,
        \ell}])
        \\
        &=
        [T_\ell(\mathbf{E}_{k, k + 1,\dots, \ell - 2}), T_\ell(\mathbf{E}_{\ell
        - 1, \ell})]
        \\
        &=
        [\mathbf{E}_{k, k + 1,\dots, \ell - 2}, E_{\ell - 1}]
        \\
        &=
        \mathbf{E}_{k, k + 1,\dots, \ell - 1}.
    \end{align*}
    \noindent Part \ref{T_E descending, remove} can be proved in a manner
    similar to part \ref{T_E ascending, remove}.

    For part \ref{T_E descending, invariant}, consider first the case when
    $\ell < k - 1$ or $\ell > m + 1$. In this setting, the result follows
    directly from \ref{Lusztig symmetries}. Next, suppose $k < \ell < m$ and
    $\ell + 1 < m$. By \ref{q-associativity},
    \[
        \mathbf{E}_{m, m - 1,\dots, k} = [[[\cdots [\mathbf{E}_{m, m - 1,\dots,
        \ell + 2}, \mathbf{E}_{\ell + 1, \ell}], E_{\ell - 1}], \cdots ], E_k].
    \]
    \noindent Thus,
    \[
        T_\ell(\mathbf{E}_{m, m - 1,\dots k}) = [[[\cdots
        [T_\ell(\mathbf{E}_{m, m - 1,\dots, \ell + 2}), T_\ell
        (\mathbf{E}_{\ell + 1, \ell}) ], T_\ell (E_{\ell - 1})], \cdots ],
        T_\ell (E_k)].
    \]
    \noindent By \ref{Lusztig symmetries}, $E_{\ell - 2}, E_{\ell - 3},\dots,
    E_k$, and $\mathbf{E}_{m, m - 1,\dots, \ell + 2}$ are fixed by $T_\ell$ and
    $T_\ell(E_{\ell - 1}) = \mathbf{E}_{\ell, \ell - 1}$, whereas by
    Proposition \ref{braid action 1}, $T_\ell([E_{\ell + 1}, E_\ell]) =
    \mathbf{T}_{\ell,\ell + 1}(E_\ell) = E_{\ell + 1}$. Hence, we obtain
    \begin{align*}
        T_\ell(\mathbf{E}_{m, m - 1,\dots, \ell})
        &=
        [[\cdots [\mathbf{E}_{m, m - 1,\dots, \ell + 1}, \mathbf{E}_{\ell, \ell
        - 1} ], \cdots ], E_k]
        \\
        &=
        \mathbf{E}_{m, m - 1,\dots, k}.
    \end{align*}
    \noindent Now suppose $k < \ell < m$ and $m = \ell + 1$.  In this case we
    have
    \begin{align*}
        T_\ell(\mathbf{E}_{m, m - 1,\dots, k})
        &=
        T_\ell([[\cdots [\mathbf{E}_{\ell + 1, \ell}, E_{\ell - 1}], \cdots ],
        E_k])
        \\
        &=
        [[\cdots [E_{\ell + 1}, \mathbf{E}_{\ell, \ell - 1}], \cdots ], E_k]
        \\
        &=
        \mathbf{E}_{m, m - 1,\dots, k}.
    \end{align*}
    \noindent Part \ref{T_E ascending, invariant} can be proved in a manner
    similar to part \ref{T_E descending, invariant}.

\end{proof}

\noindent The following lemma gives some commutation relations among certain
nested $q^{-1}$-commutators.

\begin{lemma}

    \label{EE}

    For all $1 \leq k \leq \ell < m < n$,

    \begin{enumerate}

        \item \label{EE q-commute, ascending} $\mathbf{E}_{k, k + 1, \dots,
            \ell} \mathbf{E}_{k, k + 1, \dots, m} = q\mathbf{E}_{k, k + 1,
            \dots, m} \mathbf{E}_{k, k + 1, \dots, \ell}$.

    \end{enumerate}

    For all $1 \leq k < \ell \leq m < n$,

    \begin{enumerate}

        \setcounter{enumi}{1}

        \item \label{EE q-commute, descending} $\mathbf{E}_{m, m - 1, \dots,
            \ell} \mathbf{E}_{m, m - 1, \dots, k} = q\mathbf{E}_{m, m - 1,
            \dots, k} \mathbf{E}_{m, m - 1, \dots, \ell}$.

    \end{enumerate}

    For all $1 \leq k < \ell \leq m < p < n$,

    \begin{enumerate}

        \setcounter{enumi}{2}

        \item \label{EE commute, ascending} $\mathbf{E}_{k, k + 1,\dots,
            p}\mathbf{E}_{\ell, \ell + 1,\dots, m} = \mathbf{E}_{\ell, \ell +
            1,\dots, m}\mathbf{E}_{k, k + 1,\dots, p}$,

        \item \label{EE commute, descending} $\mathbf{E}_{p, p - 1,\dots,
            k}\mathbf{E}_{m, m - 1,\dots, \ell} = \mathbf{E}_{m, m - 1,\dots,
            \ell}\mathbf{E}_{p, p - 1,\dots, k}$,

        \item \label{EE commute, ascending descending} $\mathbf{E}_{k, k +
            1,\dots, p}\mathbf{E}_{m, m - 1,\dots, \ell} = \mathbf{E}_{m, m -
            1,\dots, \ell}\mathbf{E}_{k, k + 1,\dots, p}$,

        \item \label{EE commute, descending ascending} $\mathbf{E}_{p, p -
            1,\dots, k}\mathbf{E}_{\ell, \ell + 1,\dots, m} = \mathbf{E}_{\ell,
            \ell + 1,\dots, m} \mathbf{E}_{p, p - 1,\dots, k}$.

    \end{enumerate}

    For all $1 < k < n - 1$,

    \begin{enumerate}

        \setcounter{enumi}{6}

        \item \label{EE qhat relation} $E_k\mathbf{E}_{k, k - 1, k + 1} -
            \mathbf{E}_{k, k - 1, k + 1}E_k = \widehat{q}\mathbf{E}_{k, k -
            1}\mathbf{E}_{k, k + 1}$,

        \item \label{EE qhat relation 2} $[E_{k + 1}, \mathbf{E}_{k - 1, k}]E_k
            - E_k[E_{k + 1}, \mathbf{E}_{k - 1, k}] = \widehat{q} \mathbf{E}_{k
            + 1, k} \mathbf{E}_{k - 1, k}$.

    \end{enumerate}

\end{lemma}

\begin{proof}

    To prove part \ref{EE q-commute, ascending}, we first define $\varphi :=
    \mathbf{T}_{k, k + 1,\dots, \ell - 1}^{-1}$.  From part \ref{T_E ascending,
    prepend} of Lemma \ref{T_E}, $\varphi(\mathbf{E}_{k, k + 1, \dots, \ell}) =
    E_\ell$ and $\varphi(\mathbf{E}_{k, k + 1, \dots, m}) = \mathbf{E}_{\ell,
    \ell + 1, \dots, m}$. Observe that the $q$-Serre relation \ref{q-Serre 2}
    is equivalent to $E_r\mathbf{E}_{r,s} = q\mathbf{E}_{r,s}E_r$ whenever $|r
    - s| = 1$. Thus $E_\ell\mathbf{E}_{\ell, \ell + 1} = q\mathbf{E}_{\ell,
    \ell + 1}E_\ell$.  Furthermore, by \ref{q-Serre 1}, $E_\ell$ commutes with
    $E_{\ell + 2}, E_{\ell + 3},\dots, E_m$. Therefore $E_\ell
    \mathbf{E}_{\ell, \ell + 1, \dots, m} = q \mathbf{E}_{\ell, \ell + 1,
    \dots, m} E_\ell$. Since $\varphi$ is an automorphism of ${\mathcal
    U}_q(\mathfrak{sl}(n))$, $\mathbf{E}_{k, k + 1, \dots, \ell} \mathbf{E}_{k,
    k + 1, \dots, m} = q\mathbf{E}_{k, k + 1, \dots, m} \mathbf{E}_{k, k + 1,
    \dots, \ell}$.

    For part \ref{EE q-commute, descending}, let $\varphi := \mathbf{T}_{m, m -
    1,\dots, \ell + 1}^{-1}$.  From part \ref{T_E descending, prepend} of Lemma
    \ref{T_E}, $\varphi(\mathbf{E}_{m, m - 1, \dots, \ell}) = E_\ell$ and
    $\varphi(\mathbf{E}_{m, m - 1, \dots, k}) = \mathbf{E}_{\ell, \ell - 1,
    \dots, k}$. The $q$-Serre relations \ref{q-Serre 2} and \ref{q-Serre 1}
    imply $E_\ell \mathbf{E}_{\ell, \ell - 1, \dots, k} = q \mathbf{E}_{\ell,
    \ell - 1, \dots, k} E_\ell$. Thus, $0 = \varphi^{-1}(E_\ell
    \mathbf{E}_{\ell, \ell - 1, \dots, k} - q \mathbf{E}_{\ell, \ell - 1,
    \dots, k} E_\ell) = \mathbf{E}_{m, m - 1, \dots, \ell} \mathbf{E}_{m, m -
    1, \dots, k} - q\mathbf{E}_{m, m - 1, \dots, k} \mathbf{E}_{m, m - 1,
    \dots, \ell}$.

    For part \ref{EE commute, ascending}, we first define $\varphi :=
    \mathbf{T}_{\ell, \ell + 1,\dots, m - 1}^{-1}$.  From part \ref{T_E
    descending, prepend} of Lemma \ref{T_E}, $\varphi(\mathbf{E}_{\ell, \ell +
    1,\dots, m}) = E_m$, whereas by part \ref{T_E ascending, invariant} of
    Lemma \ref{T_E}, $\varphi(\mathbf{E}_{k, k + 1,\dots, p}) = \mathbf{E}_{k ,
    k + 1,\dots, p}$. Next, define $\psi := \mathbf{T}_{k, k + 1,\dots, m
    -2}^{-1} \circ \mathbf{T}_{m + 2, m + 3,\dots, p}$. From \ref{Lusztig
    symmetries}, $\psi(E_m) = E_m$, whereas parts \ref{T_E ascending, prepend}
    and \ref{T_E ascending, remove} of Lemma \ref{T_E} imply that
    $\psi(\mathbf{E}_{k, k + 1,\dots, p}) = \mathbf{E}_{m - 1, m, m + 1}$.
    Finally we define $\theta := T_{m, m + 1}$. We observe next that
    Proposition \ref{braid action 1} implies $\theta(E_m) = E_{m + 1}$, whereas
    part \ref{T_E ascending, remove} of Lemma \ref{T_E} implies
    $\theta(\mathbf{E}_{m - 1, m, m + 1}) = E_{m - 1}$. Since $\varphi$,
    $\psi$, and $\theta$ are automorphisms of ${\mathcal
    U}_q(\mathfrak{sl}(n))$, the composition $\theta \circ \psi \circ \varphi$
    is also an automorphism of ${\mathcal U}_q(\mathfrak{sl}(n))$. Since $E_{m
    + 1}$ and $E_{m - 1}$ commute and are the images of $\mathbf{E}_{\ell, \ell
    + 1,\dots, m}$ and $\mathbf{E}_{k, k + 1,\dots, p}$ respectively under
    $\theta \circ \psi \circ \varphi$, this implies that $\mathbf{E}_{\ell,
    \ell + 1,\dots, m}$ and $\mathbf{E}_{k, k + 1,\dots, p}$ also commute.
    Parts \ref{EE commute, descending}, \ref{EE commute, ascending descending},
    and \ref{EE commute, descending ascending} can be proved similarly.

    To prove part \ref{EE qhat relation}, we first use the $q$-Serre relation
    \ref{q-Serre 2}, which is equivalent to $E_k\mathbf{E}_{k, k - 1} =
    q\mathbf{E}_{k, k - 1}E_k$ to get
    \begin{align*}
        E_k\mathbf{E}_{k, k - 1, k + 1}
        &=
        E_k(\mathbf{E}_{k, k -1}E_{k + 1} - q^{-1}E_{k + 1}\mathbf{E}_{k, k -
        1})
        \\
        &=
        q\mathbf{E}_{k, k -1}E_kE_{k + 1} - q^{-1}E_kE_{k + 1}\mathbf{E}_{k, k
        - 1}.
    \end{align*}
    \noindent Next we make the substitution $E_kE_{k + 1} = \mathbf{E}_{k, k +
    1} + q^{-1}E_{k + 1}E_k$ to obtain
    \begin{equation}
        \begin{aligned}
            \label{EE, part 7}
            E_k\mathbf{E}_{k, k - 1, k + 1}
            &=
            q\mathbf{E}_{k, k - 1}\mathbf{E}_{k, k + 1} - q^{-1}\mathbf{E}_{k,
            k + 1}\mathbf{E}_{k, k - 1}
            \\
            &
            \hspace{5mm} + \mathbf{E}_{k, k - 1}E_{k + 1}E_k - q^{-2}E_{k +
            1}E_k\mathbf{E}_{k, k - 1}.
        \end{aligned}
    \end{equation}
    \noindent Since $\mathbf{E}_{k, k - 1} = T_k(E_{k - 1})$, $\mathbf{E}_{k, k
    + 1} = T_k(E_{k + 1})$ and $E_{k - 1}$ and $E_{k + 1}$ commute, this
    implies that $\mathbf{E}_{k, k - 1}$ and $\mathbf{E}_{k, k + 1}$ commute
    also.  Hence, the first two terms in \ref{EE, part 7} above involving the
    products $\mathbf{E}_{k, k - 1}\mathbf{E}_{k, k + 1}$ and $\mathbf{E}_{k, k
    + 1}\mathbf{E}_{k, k - 1}$ can be combined to get $E_k\mathbf{E}_{k, k - 1,
    k + 1} = \widehat{q}\mathbf{E}_{k, k - 1}\mathbf{E}_{k, k + 1} +
    \mathbf{E}_{k, k - 1}E_{k + 1}E_k - q^{-2}E_{k + 1}E_k\mathbf{E}_{k, k -
    1}$. We again use $E_k\mathbf{E}_{k, k - 1} = q\mathbf{E}_{k, k - 1}E_k$ to
    obtain
    \begin{align*}
        E_k\mathbf{E}_{k, k - 1, k + 1}
        &=
        \widehat{q}\mathbf{E}_{k, k - 1}\mathbf{E}_{k, k + 1} + \mathbf{E}_{k,
        k - 1}E_{k + 1}E_k - q^{-1}E_{k + 1}\mathbf{E}_{k, k - 1}E_k
        \\
        &=
        \widehat{q}\mathbf{E}_{k, k - 1}\mathbf{E}_{k, k + 1} + \mathbf{E}_{k,
        k - 1, k + 1}E_k.
    \end{align*}
    To prove part \ref{EE qhat relation 2}, we first use the $q$-Serre relation
    \ref{q-Serre 2}, which is equivalent to $E_k\mathbf{E}_{k - 1, k} =
    q^{-1}\mathbf{E}_{k - 1, k}E_k$ to get
    \begin{align*}
        E_k[E_{k + 1}, \mathbf{E}_{k - 1, k}]
        &=
        E_k (E_{k + 1}\mathbf{E}_{k - 1, k} - q^{-1}\mathbf{E}_{k - 1, k}E_{k +
        1})
        \\
        &=
        E_kE_{k + 1}\mathbf{E}_{k - 1, k} - q^{-2}\mathbf{E}_{k - 1, k}E_kE_{k
        + 1}.
    \end{align*}
    \noindent Next we substitute $E_kE_{k + 1}$ with $-q\mathbf{E}_{k + 1, k}
    +qE_{k + 1}E_k$ to obtain
    \begin{equation}
        \begin{aligned}
            \label{EE, part 8}
            E_k[E_{k + 1}, \mathbf{E}_{k - 1, k}]
            &=
            -q\mathbf{E}_{k + 1, k}\mathbf{E}_{k - 1, k} + q^{-1}\mathbf{E}_{k
            - 1, k}\mathbf{E}_{k + 1, k}
            \\
            &
            \hspace{5mm}- q^{-1}\mathbf{E}_{k - 1, k}E_{k + 1}E_k + E_{k +
            1}\mathbf{E}_{k - 1, k}E_k.
        \end{aligned}
    \end{equation}
    \noindent Observe that $\mathbf{E}_{k - 1, k}$ and $\mathbf{E}_{k + 1, k}$
    commute. Hence the first two terms in \ref{EE, part 8} above involving the
    products $\mathbf{E}_{k - 1, k}\mathbf{E}_{k + 1, k}$ and $\mathbf{E}_{k +
    1, k}\mathbf{E}_{k - 1, k}$ can be combined to get
    \[
        E_k[E_{k + 1}, \mathbf{E}_{k - 1, k}] = -\widehat{q}\mathbf{E}_{k + 1,
        k}\mathbf{E}_{k - 1, k} - q^{-1}\mathbf{E}_{k - 1, k}E_{k + 1}E_k +
        E_{k + 1} \mathbf{E}_{k - 1, k}E_k,
    \]
    \noindent which is equivalent to the identity as written in part \ref{EE
    qhat relation 2}.

\end{proof}

Recall the definition of the function $r:[n] \to J \cup \{0\}$,
\[
    r(m) := \operatorname{max}\{k \in J \cup \{0\} \mid k < m\}.
\]
\noindent The following lemma tells us how each root vector in
${\mathcal U}_q(\mathfrak{n}_J)$ can be written as a nested $q^{-1}$-commutator.

\begin{lemma}
    \label{lemma 1}
    If $(i, j) \in \Phi_J$, then
    \[
        X_{ij} = {\mathbf E}_{r(j), r(j) - 1, \cdots, i, r(j) + 1, r(j) + 2,
        \dots, j - 1}.
    \]

\end{lemma}

\begin{proof}

    Let $J = \{i_1 < \cdots < i_t\} \subseteq [n - 1]$.  For $k \in [t]$, let
    $w_k\in W$ be the initial segment of $w_J$ defined as $w_k := S_t S_{t-1}
    \cdots S_{k+1} \in W$ (recall the definition of $S_k$ in \ref{S_k's}).
    Observe that $w_k(p) = p$ whenever $p \leq i_k$, whereas $w_k(p) = w_J(p -
    i_k)$ if $p > i_k$.

    For $p \leq n$, define $M(p) := \min\{r\in J \cup \{n\} \mid p \leq r\}$.
    Assume $(i, j) \in \Phi_J$. We use the abbreviation $\ell(i, j) = i +
    w_J^{-1}(j) - 1$. The restrictions imposed on $i$ and $j$ imply that
    $\ell(i, j) < n$. Define the Weyl group elements $v_{ij} :=
    s_is_{i+1}\cdots s_{\ell(i, j) - 1} \in W$ and $ u_i := (s_{M(i)}\cdots
    s_{n-1})(s_{M(i) - 1}\cdots s_{n-2})\cdots (s_{i+1}\cdots s_{n - M(i) + i})
    \in W$.  Let $N(p) := 1 + \#\{r\in J \mid r < p\}$ and let $W_{ij}\in W$ be
    the initial segment of $w_J$ defined as $W_{ij} : = w_{N(i)} u_i v_{ij}$.
    Since $W_{ij} s_{\ell(i, j)}$ is also an initial segment of $w_J$, it
    follows that $W_{ij}(\alpha_{\ell(i, j)})$ is a positive Lusztig root.  In
    fact, we have
    \[
        W_{ij}(\alpha_{\ell(i, j)}) = w_{N(i)}u_i(e_i - e_{\ell(i, j) + 1}) =
        w_{N(i)} (e_i - e_{M(i) + w_J^{-1}(j)}) = e_i - e_j.
    \]
    \noindent Therefore, the simple root $\alpha_{i_k}$ is a Lusztig root
    because $W_{i_k, i_k + 1}(\alpha_{\ell(i_k, i_k + 1)}) = e_{i_k} - e_{i_k +
    1} = \alpha_{i_k}$.  Proposition \ref{braid action 1} implies that
    $T_{W_{i_k, i_k + 1}}(E_{\ell(i_k, i_k + 1)}) = E_{i_k}$. Thus, $X_{i_k,
    i_k + 1} = E_{i_k}$.  More generally, $X_{ij} = T_{W_{ij}}(E_{\ell(i,
    j)})$.

    We assume now that $i > 1$ and $(i, j) \in \Phi_J$. Define the Weyl
    group elements $y_{ij} : = (s_{\ell(i, j)} \cdots s_{n - M(i) + i - 1})
    v_{i - 1, j}$. We have $W_{i - 1, j} = W_{ij}y_{ij}$.  Thus, $X_{i - 1, j}
    = T_{w_{N(i)}} T_{u_i} T_{v_{ij}} T_{y_{ij}} (E_{\ell(i, j) - 1})$.  Since
    \[
        T_{y_{ij}}(E_{\ell(i, j) - 1}) = T_{v_{i - 1, j}} T_{\ell(i, j)}
        (E_{\ell(i, j) - 1}) = T_{v_{i - 1, j}} ([E_{\ell(i, j)}, E_{\ell(i, j)
        - 1}])
    \]
    \noindent and the braid group generators act via algebra automorphisms, we
    have $X_{i - 1, j} = T_{w_{N(i)}} T_{u_i} T_{v_{ij}} ([T_{v_{i - 1, j}}
    (E_{\ell(i, j)}), T_{v_{i - 1, j}} (E_{\ell(i, j) - 1})])$.  However, since
    $v_{i -1, j} (\alpha_{\ell(i, j)}) = \alpha_{\ell(i, j)}$ and $v_{ij} v_{i
    - 1, j} (\alpha_{\ell(i, j) - 1}) = \alpha_{i - 1}$, Proposition \ref{braid
    action 1} implies
    \[
        X_{i - 1, j} = T_{w_{N(i)}} T_{u_i} ([T_{v_{ij}} (E_{\ell(i, j)}), E_{i
        - 1}]).
    \]
    \noindent Finally, since $w_{N(i)} u_i (\alpha_{i - 1}) = \alpha_{i - 1}$,
    we have $X_{i - 1, j} = [X_{ij}, E_{i - 1}]$.  Hence, for all $r \in J$ and
    $1 \leq i \leq r$,  we iteratively get $X_{i, r + 1} = \mathbf{E}_{r, r -
    1, \dots, i}$.

    Next suppose $i$ and $j$ are a pair of integers in $[n - 1]$ such that $j
    \not\in J$ and $(i, j) \in \Phi_J$. Hence, $w_J^{-1}(j + 1) =
    w_J^{-1}(j) + 1$, $\ell(i, j + 1) = \ell(i, j) + 1$, and $W_{i, j + 1} =
    W_{ij}s_{\ell(i, j)}$. Hence,
    \[
        X_{i, j + 1} = T_{W_{i, j + 1}}(E_{\ell(i, j + 1)}) =
        T_{W_{ij}}T_{s_{\ell(i, j)}}(E_{\ell(i, j) + 1}) =
        T_{W_{ij}}([E_{\ell(i, j)}, E_{\ell(i, j) + 1}]).
    \]
    \noindent However, since
    \begin{align*}
        W_{ij}(\alpha_{\ell(i, j) + 1})
        &=
        w_{N(i)} u_i v_{ij} (\alpha_{\ell(i, j) + 1})
        \\
        &=
        w_{N(i)} u_i (\alpha_{\ell(i, j) + 1})
        \\
        &=
        w_{N(i)}(\alpha_{M(i) + w_J^{-1}(j)})
        \\
        &=
        \alpha_j,
    \end{align*}
    \noindent it follows from Proposition\ref{braid action 1} that
    $T_{W_{ij}}(E_{\ell(i, j) + 1}) = E_j$.  Thus,
    \begin{align*}
        X_{i, j + 1}
        &=
        T_{W_{ij}}([E_{\ell(i, j)}, E_{\ell(i, j) + 1}])
        \\
        &=
        [T_{W_{ij}}(E_{\ell(i, j)}), T_{W_{ij}}(E_{\ell(i, j) + 1})]
        \\
        &=
        [X_{ij}, E_j].
    \end{align*}
    \noindent If $r \in J$ and $r + 1,\dots r + s \not\in J$, we iteratively
    obtain
    \begin{align*}
        X_{i, r + s + 1}
        &=
        [[[ \cdots [X_{i, r + 1}, E_{r + 1}], E_{r + 2}], \cdots ], E_{r + s}]
        \\
        &=
        [[[ \cdots [\mathbf{E}_{r, r - 1, \dots, i}, E_{r + 1}], E_{r + 2}],
        \cdots ], E_{r + s}]
        \\
        &=
        \mathbf{E}_{r, r - 1, \dots, i, r + 1, r + 2, \dots, r + s}.
    \end{align*}

\end{proof}

\noindent The next lemma tells us how Lusztig's symmetries act on the root
vectors of ${\mathcal U}_q(\mathfrak{n}_J)$.

\begin{lemma}

    \label{T_X}

    Suppose $(i, j) \in \Phi_J$.

    \begin{enumerate}

        \item \label{T_X column drop}If $j > r(j) + 1$, then $T_{j - 1}(X_{ij})
            = X_{i, j - 1}$.

        \item \label{T_X row raise} If $i < r(j)$, then $T_i(X_{ij}) = X_{i +
            1, j}$.

        \item \label{T_X invariant} If $k \in [n - 1]$ and $k \not\in \{i - 1,
            i, r(j), j - 1, j\}$, $T_k(X_{ij}) = X_{ij}$.

    \end{enumerate}

\end{lemma}

\begin{proof}

    For short, let $r = r(j)$. To prove part \ref{T_X column drop}, we first
    consider the case when $j > r + 2$. In this setting we can use the
    associativity property \ref{q-associativity} to write $X_{ij} =
    [\mathbf{E}_{r, r - 1,\dots, i, r + 1, r + 2,\dots, j - 3}, \mathbf{E}_{j -
    2, j - 1}]$. From \ref{Lusztig symmetries}, $\mathbf{E}_{r, r - 1,\dots, i,
    r + 1, r + 2,\dots, j - 3}$ is fixed by $T_{j - 1}$, and by Proposition
    \ref{braid action 1}, $T_{j-1} (\mathbf{E}_{j - 2, j - 1}) = \mathbf{T}_{j
    - 1, j - 2}(E_{j - 1}) = E_{j - 2}$. Hence $T_{j - 1}(X_{ij}) =
    [\mathbf{E}_{r, r - 1,\dots, i, r + 1, r + 2,\dots, j - 3}, E_{j - 2}] =
    X_{i, j - 1}$.  On the other hand, if $j = r + 2$, we can use
    \ref{q-commutativity} to rewrite $X_{ij}$ as $X_{ij} = \mathbf{E}_{r, r -
    1, \dots, i, r + 1} = \mathbf{E}_{r, r + 1, r - 1,\dots, i}$. By
    Proposition \ref{braid action 1}, $T_{j - 1}(\mathbf{E}_{r, r + 1}) = E_r$,
    whereas by part \ref{T_E descending, invariant} of Lemma \ref{T_E}, we have
    $T_{j - 1}(E_{r - 1, r - 2,\dots, i}) = E_{r - 1, r - 2,\dots, i}$.
    Therefore $T_{j - 1}(X_{ij}) = \mathbf{E}_{r, r - 1, \dots, i} = X_{i, j -
    1}$.

    In proving part \ref{T_X row raise} we first suppose $i + 1 < r$. We can
    use the associativity property \ref{q-associativity} to write $X_{ij} =
    [[[[ \cdots [\mathbf{E}_{r, \dots, i + 2}, \mathbf{E}_{i + 1, i}], E_{r +
    1}], E_{r + 2}], \cdots], E_{j - 1}]$.  By \ref{Lusztig symmetries},
    $T_i(E_k) = E_k$ ($k = r + 1, r + 2, \dots, j - 1$), by Proposition
    \ref{braid action 1}, $T_i(\mathbf{E}_{i + 1, i}) = E_{i + 1}$, and by part
    \ref{T_E descending, invariant} of Lemma \ref{T_E}, $T_i(\mathbf{E}_{r,
    \dots, i + 2}) = \mathbf{E}_{r, \dots, i + 2}$.  Hence $T_i(X_{ij}) = [[[
        \cdots [\mathbf{E}_{r, \dots, i + 1}, E_{r + 1}],  E_{r + 2}], \cdots
    ], E_{j - 1}] = X_{i + 1, j}$.  On the other hand, if $i + 1 = r$ then
    $X_{ij}$ can be written as $X_{ij} = \mathbf{E}_{i + 1, i, r + 1, r + 2,
    \dots, j - 1}$. In this case the result follows, again, by using $E_k$ ($k
    = r + 1, r + 2,\dots, j - 1$) is fixed by $T_i$ and $T_i(\mathbf{E}_{i + 1,
    i}) = E_{i + 1}$.

    For part \ref{T_X invariant}, we consider first the case when $k < i - 1$
    or $k > j$. Here, the result follows directly from the definition of the
    Lusztig symmetries \ref{Lusztig symmetries}. Now suppose $i < k < r(j)$.
    By part \ref{T_E descending, invariant} of Lemma \ref{T_E},
    $\mathbf{E}_{r(j), r(j) - 1,\dots, i}$ is fixed by the automorphism $T_k$.
    Furthermore, by \ref{Lusztig symmetries}, $E_{r(j) + 1}, E_{r(j) +
    2},\dots, E_{j - 1}$ are also fixed by $T_k$. Thus $X_{ij}$, which can be
    written as $[[[ \cdots [\mathbf{E}_{r(j), r(j) - 1,\dots, i}, E_{r(j) +
    1}], E_{r(j) + 2}], \cdots ],  E_{j - 1}]$, is also fixed by $T_k$.
    Finally suppose $r(j) < k < j - 1$.  We can use \ref{q-commutativity} to
    write $X_{ij}$ as $[[[ \cdots [\mathbf{E}_{r(j), r(j) + 1, \dots, j - 1},
    E_{r(j) - 1}], E_{r(j) - 2}],\cdots ], E_i]$.  From part \ref{T_E
    ascending, invariant} of Lemma \ref{T_E}, $\mathbf{E}_{r(j), r(j) +
    1,\dots, j - 1}$ is fixed by the automorphism $T_k$, while by \ref{Lusztig
    symmetries}, $E_{r(j) - 1}$, $E_{r(j) - 2},\dots, E_i$ are fixed by $T_k$.
    Hence, $X_{ij}$ is also fixed by $T_k$.

\end{proof}

\noindent The following theorem is the main result of this section. It gives
the defining relations in ${\mathcal U}_q(\mathfrak{n}_J)$.

\begin{theorem}

    The quantized nilradical ${\mathcal U}_q(\mathfrak{n}_J)$ is generated by
    the Lusztig root vectors $X_{ij}$ $((i, j) \in \Phi_J)$ and has the
    following defining relations:
    \begin{align}
        X_{ij}X_{\ell m} &=
            \begin{cases}
                qX_{\ell m}X_{ij}
                &
                $\begin{tabularx}{\textwidth}{l}
                    $(\ell < i \text{ and } j = m)$
                    \\
                    or
                    $(\ell = i \text{ and } w_J^{-1}(j) <
                    w_J^{-1}(m))$
                \end{tabularx}$
                \\
                \\
                X_{\ell m}X_{ij}
                &
                $\begin{tabularx}{\textwidth}{l}
                    $(\ell < i < w_0^J(j) < w_0^J(m))$
                    \\
                    or
                    $(\ell < m < i < j)$
                    \\
                    or
                    $(\ell \leq r(m) < i < m < j)$
                \end{tabularx}$
                \\
                \\
                X_{\ell m}X_{ij} + \widehat{q}X_{\ell j}X_{im}
                &
                $\begin{tabularx}{\textwidth}{l}
                    $(\ell < i \leq r(m) < j < m)$
                    \\
                    or
                    $(\ell < i \leq r(m) < m \leq r(j) < j)$
                \end{tabularx}$
                \\
                \\
                q^{-1}X_{\ell m}X_{ij} + X_{(\ell j), m}
                &
                $\begin{tabularx}{\textwidth}{l}
                    $(\ell < m = i < j)$
                \end{tabularx}$
            \end{cases}
        \end{align}
        \noindent where $X_{(\ell j), m} := (-q)^{m - r(m) - 1} X_{\ell j} +
        \widehat{q}\sum_{r(m) < k < m}(-q)^{m - k - 1} X_{kj} X_{\ell k}$.

\end{theorem}

\begin{proof}

    Suppose first $\ell < i$ and $j = m$. Let $r = r(j)$ and let $\varphi =
    \mathbf{T}_{r + 1, \dots, j - 2, j - 1}$.  By part \ref{T_X column drop} of
    Lemma \ref{T_X}, $\varphi (X_{ij}) = X_{i, r + 1}$ and $\varphi (X_{\ell
    m}) = X_{\ell, r + 1}$, and from Lemma \ref{lemma 1}, $X_{i, r + 1} =
    \mathbf{E}_{r, r - 1, \dots, i}$ and $X_{\ell, r + 1} = \mathbf{E}_{r, r -
    1, \dots, \ell}$. By Lemma \ref{EE}, part \ref{EE q-commute, descending},
    $\varphi(X_{ij}) \varphi(X_{\ell m}) = q \varphi(X_{\ell m})
    \varphi(X_{ij})$. Since $\varphi$ is an automorphism of ${\mathcal
    U}_q(\mathfrak{sl}(n))$, $X_{ij}X_{\ell m} = qX_{\ell m} X_{ij}$.

    Now suppose $\ell = i$ and $w_{J}^{-1}(j) < w_J^{-1}(m)$. If $r(j) = r(m) =
    r$, then $j < m$. In this case, let $\varphi = \mathbf{T}_{r - 1, r -
    2,\dots, i}$.  By part \ref{T_X row raise} of Lemma \ref{T_X},
    $\varphi(X_{ij}) = X_{rj}$ and $\varphi(X_{\ell m}) = X_{rm}$, and from
    Lemma \ref{lemma 1}, $X_{rj} = \mathbf{E}_{r, r + 1,\dots, j - 1}$ and
    $X_{rm} = \mathbf{E}_{r, r + 1,\dots, m - 1}$. Part \ref{EE q-commute,
    ascending} of Lemma \ref{EE} implies $\varphi(X_{ij})\varphi(X_{\ell m}) =
    q\varphi(X_{\ell m})\varphi(X_{ij})$.  Therefore $X_{ij}X_{\ell m} = q
    X_{\ell m} X_{ij}$.  On the other hand, if $r(j) \neq r(m)$, then $r(m) < m
    \leq r(j) < j$.  Let
    \[
        \psi_1 := \mathbf{T}_{r(m) + 1, r(m) + 2,\dots, m - 1} \circ
        \mathbf{T}_{r(j) + 1, r(j) + 2,\dots, j - 1} \circ \mathbf{T}_{r(m) -
        1,\dots, i + 1, i}.
    \]
    \noindent From Lemma \ref{T_X}, $\psi_1(X_{ij}) = X_{r(m), r(j) + 1}$ and
    $\psi_1(X_{\ell m}) = X_{r(m), r(m) + 1}$.  By Lemma \ref{lemma 1},
    $X_{r(m), r(j) + 1} = \mathbf{E}_{r(j), r(j) - 1,\dots, r(m)}$ and
    $X_{r(m), r(m) + 1} = E_{r(m)}$. It follows from the $q$-commutativity
    relation (part \ref{EE q-commute, descending} of Lemma \ref{EE}) that
    $X_{r(m), r(j) + 1}X_{r(m), r(m) + 1} = qX_{r(m), r(m) + 1}X_{r(m), r(j) +
    1}$. Hence, $X_{ij}X_{\ell m} = qX_{\ell m} X_{ij}$.

    Now suppose $\ell < i < w_0^J(j) < w_0^J(m)$. If $r(j) = r(m) = r$, then we
    must have $\ell < i \leq r < m < j$. Let
    \[
        \psi_2 := \mathbf{T}_{r - 2, r - 3,\dots, i - 1} \circ \mathbf{T}_{r +
        1, r + 2,\dots, m - 1} \circ \mathbf{T}_{r - 1, r - 2,\dots, i} \circ
        \mathbf{T}_{m + 1, m + 2,\dots, j - 1} \circ \mathbf{T}_{i - 2, i - 3,
        \dots, \ell}.
    \]
    \noindent By Lemmas \ref{lemma 1} and \ref{T_X},
    \begin{align*}
        &\psi_2(X_{ij}) = X_{r, m + 1} = \mathbf{E}_{r, r + 1,\dots, m},
        \\
        &\psi_2(X_{\ell m}) = X_{r - 1, r + 1} = \mathbf{E}_{r, r - 1}.
    \end{align*}
    \noindent Hence, from part \ref{T_E ascending, prepend} of Lemma \ref{T_E},
    $(T_r^{-1} \circ \psi_2)(X_{ij}) = \mathbf{E}_{r + 1, r + 2,\dots, m}$.
    Furthermore $(T_r^{-1} \circ \psi_2)(X_{\ell m}) = E_{r - 1}$.  However,
    since $\mathbf{E}_{r + 1, r + 2,\dots, m}$ commutes with $E_{r - 1}$, it
    follows that $X_{ij}$ and $X_{\ell m}$ must also commute. On the other
    hand, if $r(j) \neq r(m)$, then we must have $\ell < i \leq r(j) < j \leq
    r(m) < m$.  Let
    \[
        \psi_3 := \mathbf{T}_{r(j) - 1, r(j) - 2,\dots, i} \circ
        \mathbf{T}_{r(j) + 1, r(j) + 2,\dots, j - 1} \circ \mathbf{T}_{i - 2, i
        - 3,\dots, \ell} \circ \mathbf{T}_{r(m) + 1, r(m) + 2,\dots, m - 1}.
    \]
    \noindent By Lemmas \ref{lemma 1} and \ref{T_X},
    \begin{align*}
        &\psi_3(X_{ij}) = X_{r(j), r(j) + 1} = E_{r(j)},
        \\
        &\psi_3(X_{\ell m}) = X_{i - 1, r(m) + 1} = \mathbf{E}_{r(m), r(m) -
        1,\dots, i - 1}.
    \end{align*}
    \noindent By part \ref{EE commute, descending} of Lemma \ref{EE},
    $E_{r(j)}$ and $\mathbf{E}_{r(m), r(m) - 1, \dots, i - 1}$ commute. Hence,
    $X_{ij}$ and $X_{\ell m}$ commute also.

    Next suppose $\ell < m < i < j$. Since $X_{ij} = \mathbf{E}_{r(j), r(j) -
    1,\dots, i, r(j) + 1, r(j) + 2,\dots, j - 1}$ and $X_{\ell m} =
    \mathbf{E}_{r(m), r(m) - 1,\dots, \ell, r(m) + 1, r(m) + 2,\dots, m - 1}$
    (Lemma \ref{lemma 1}) and each of $E_i,\dots, E_{j - 1}$ commutes with each
    of $E_{\ell},\dots, E_{m - 1}$ (\ref{q-Serre 1}), it follows that $X_{ij}$
    commutes with $X_{\ell m}$.

    Next suppose $\ell \leq r(m) < i < m < j$. Therefore, $\ell \leq r(i) =
    r(m) < i < m \leq r(j) < j$.  Let
    \begin{equation*}
        \begin{aligned}
            \psi_4 &:= \mathbf{T}_{r(j) - 1, r(j) - 2,\dots, m} \circ
            \mathbf{T}_{r(m) + 1, r(m) + 2,\dots, m - 1}
            \\
            &
            \hspace{6mm}\circ \mathbf{T}_{m - 2, m - 3,\dots, i} \circ
            \mathbf{T}_{r(m) - 1, r(m) - 2,\dots, \ell} \circ \mathbf{T}_{r(j)
            + 1, r(j) + 2,\dots, j - 1}.
        \end{aligned}
    \end{equation*}
    \noindent From Lemmas \ref{lemma 1} and  \ref{T_X},
    \begin{align*}
        &\psi_4(X_{ij}) = X_{r(j), r(j) + 1} = E_{r(j)},
        \\
        &\psi_4(X_{\ell m}) = X_{r(m), r(m) + 1} = E_{r(m)}.
    \end{align*}
    \noindent Since $r(j) - r(m) > 1$, the $q$-Serre relation \ref{q-Serre 1}
    implies that $E_{r(j)}$ and $E_{r(m)}$ commute. Hence $X_{ij}$ and $X_{\ell
    m}$ commute also.

    Next suppose $\ell < i \leq r(m) < j < m$. Thus $r(j) = r(m) = r$.  Let
    \begin{equation*}
        \begin{aligned}
            \psi_5
            &=
            \mathbf{T}_{r + 2, r + 3,\dots, j} \circ \mathbf{T}_{r - 2, r -
            3,\dots, i - 1} \circ \mathbf{T}_{r + 1, r + 2,\dots, j -1}
            \\
            &
            \hspace{6mm}\circ \mathbf{T}_{r - 1, r - 2,\dots, i} \circ
            \mathbf{T}_{j + 1, j + 2, \dots, m - 1} \circ \mathbf{T}_{i - 2, i
            - 3,\dots, \ell}.
        \end{aligned}
    \end{equation*}
    \noindent By Lemmas \ref{lemma 1} and \ref{T_X},
    \begin{align*}
        &\psi_5(X_{ij}) = X_{r, r + 1} = E_r,
        \\
        &\psi_5(X_{\ell m}) = X_{r - 1, r + 2} = \mathbf{E}_{r, r - 1, r + 1},
        \\
        &\psi_5(X_{\ell j}) = X_{r - 1, r + 1} = \mathbf{E}_{r, r - 1},
        \\
        &\psi_5(X_{im}) = X_{r, r + 2} = \mathbf{E}_{r, r + 1}.
    \end{align*}
    \noindent Part \ref{EE qhat relation} of Lemma \ref{EE} implies
    \[
        \psi_5(X_{ij}) \psi_5(X_{\ell m}) = \psi_5(X_{\ell m}) \psi_5(X_{ij}) +
        \widehat{q} \psi_5(X_{\ell j}) \psi_5(X_{im}).
    \]
    \noindent Since $\psi_5$ is an automorphism of ${\mathcal
    U}_q(\mathfrak{sl}(n))$, $X_{ij}X_{\ell m} = X_{\ell m}X_{ij} +
    \widehat{q}X_{\ell j}X_{im}$.

    Next suppose $\ell < i \leq r(m) < m \leq r(j) < j$.  Let
    \begin{equation*}
        \begin{aligned}
            \psi_6
            &=
            \mathbf{T}_{r(m) - 2, r(m) - 3,\dots, i - 1} \circ \mathbf{T}_{r(m)
            - 1, r(m) - 2,\dots, i}
            \\
            &
            \hspace{6mm}\circ \mathbf{T}_{r(m) + 1, r(m) + 2,\dots, m - 1}
            \circ \mathbf{T}_{r(j) + 1, r(j) + 2,\dots, j - 1} \circ
            \mathbf{T}_{i - 2, i - 3,\dots, \ell}.
        \end{aligned}
    \end{equation*}
    \noindent By Lemmas \ref{lemma 1} and \ref{T_X},
    \begin{align*}
        &\psi_6(X_{ij}) = X_{r(m), r(j) + 1} = \mathbf{E}_{r(j), r(j) -
        1,\dots, r(m)},
        \\
        &\psi_6(X_{\ell m}) = X_{r(m) - 1, r(m) + 1} = \mathbf{E}_{r(m), r(m) -
        1},
        \\
        &\psi_6(X_{\ell j}) = X_{r(m) - 1, r(j) + 1} = \mathbf{E}_{r(j), r(j) -
        1,\dots, r(m) - 1},
        \\
        &\psi_6(X_{im}) = X_{r(m), r(m) + 1} = E_{r(m)}.
    \end{align*}
    \noindent Next let $\xi = \mathbf{T}_{r(j), r(j) - 1,\dots, r(m) +
    2}^{-1}$.  Lemma \ref{T_E} implies
    \begin{align*}
        &(\xi \circ \psi_6)(X_{ij}) = \mathbf{E}_{r(m) + 1, r(m)},
        \\
        &(\xi \circ \psi_6)(X_{\ell j}) = \mathbf{E}_{r(m) + 1, r(m), r(m)
        - 1},
    \end{align*}
    \noindent whereas \ref{Lusztig symmetries} implies $(\xi \circ
    \psi_6)(X_{\ell m}) = \mathbf{E}_{r(m), r(m) - 1}$ and $(\xi \circ
    \psi_6)(X_{im}) = E_{r(m)}$. From Proposition \ref{braid action 1}, observe
    $T_{r(m) - 1}(\mathbf{E}_{r(m), r(m) - 1}) = E_{r(m)}$, whereas by
    \ref{Lusztig symmetries},
    \begin{align*}
        &T_{r(m) - 1}(\mathbf{E}_{r(m) + 1, r(m)}) = [E_{r(m) + 1},
        \mathbf{E}_{r(m) - 1, r(m)}],
        \\
        &T_{r(m) - 1}(E_{r(m)}) = \mathbf{E}_{r(m) - 1, r(m)}.
    \end{align*}
    \noindent From part \ref{T_E descending, remove} of Lemma \ref{T_E},
    $T_{r(m) - 1}(\mathbf{E}_{r(m) + 1, r(m), r(m) - 1}) = \mathbf{E}_{r(m) +
    1, r(m)}$.  Since $T_{r(m) - 1}$, $\xi$, and $\psi_6$ are automorphisms
    of ${\mathcal U}_q(\mathfrak{sl}(n))$, the composition $\theta : = T_{r(m)
    - 1} \circ \xi \circ \psi_6$ is also an automorphism of ${\mathcal
    U}_q(\mathfrak{sl}(n))$.  Since
    \begin{align*}
        &\theta(X_{ij}) = [E_{r(m) + 1}, \mathbf{E}_{r(m) - 1, r(m)}],
        \\
        &\theta(X_{\ell m}) = \mathbf{E}_{r(m)},
        \\
        &\theta(X_{\ell j}) = \mathbf{E}_{r(m)+ 1, r(m)},
        \\
        &\theta(X_{im}) = \mathbf{E}_{r(m) - 1, r(m)},
    \end{align*}
    \noindent it follows from part \ref{EE qhat relation 2} of Lemma \ref{EE}
    that $X_{ij}X_{\ell m} = X_{\ell m} X_{ij} + \widehat{q} X_{\ell j}
    X_{im}$.

    Finally suppose $\ell < m = i < j$.  Let
    \[
        \psi_7 = \mathbf{T}_{r(j) + 1, r(j) + 2,\dots, j - 1} \circ
        \mathbf{T}_{r(m) - 1, r(m) - 2,\dots, \ell}.
    \]
    \noindent By Lemmas \ref{lemma 1} and \ref{T_X},
    \begin{align*}
        &\psi_7(X_{ij}) = X_{i, r(j) + 1} = \mathbf{E}_{r(j), r(j) - 1,\dots,
        i},
        \\
        &\psi_7(X_{\ell m}) = X_{r(m), m} = \mathbf{E}_{r(m), r(m) + 1,\dots, m
        - 1},
        \\
        &\psi_7(X_{\ell j}) = X_{r(m), r(j) + 1} = \mathbf{E}_{r(j), r(j) -
        1,\dots, r(m)}.
    \end{align*}
    \noindent We proceed now by induction on $m - r(m)$. First suppose $m -
    r(m) = 1$.  Hence $\psi_7(X_{\ell m}) = E_{r(m)}$ and we have
    \begin{align*}
        \psi_7(X_{ij} X_{\ell m} - q^{-1}X_{\ell m} X_{ij})
        &=
        \mathbf{E}_{r(j), r(j) - 1,\dots, i}E_{r(m)} - q^{-1} E_{r(m)}
        \mathbf{E}_{r(j), r(j) - 1,\dots, i}
        \\
        &= \mathbf{E}_{r(j), r(j) - 1,\dots, i - 1}
        \\
        &= X_{i - 1, r(j) + 1}
        \\
        &= \psi_7(X_{\ell j}).
    \end{align*}
    \noindent Next, suppose $m - r(m) > 1$. From Lemmas \ref{lemma 1} and
    \ref{T_X},
    \begin{align*}
        &\psi_7(X_{i - 1, j}) = X_{i - 1, r(j) + 1} = \mathbf{E}_{r(j), r(j) -
        1,\dots, i - 1}
        \\
        &\psi_7(X_{\ell, m - 1}) = X_{r(m), m - 1} = \mathbf{E}_{r(m), r(m) +
        1,\dots, m - 2}.
    \end{align*}
    \noindent Therefore,
    \begin{align*}
        \psi_7([X_{ij}, X_{\ell m}])
        &= \mathbf{T}_{r(m), r(m) + 1,\dots, m - 2}([X_{i, r(j) + 1}, E_{m -
        1}])
        \\
        &= \mathbf{T}_{r(m), r(m) + 1,\dots, m - 2}(\mathbf{E}_{r(j), r(j) -
        1,\dots, m - 1})
        \\
        &= \mathbf{T}_{r(m), r(m) + 1,\dots, m - 2} \circ \mathbf{T}_{r(j),
        r(j) - 1,\dots, m} (E_{m - 1})
        \\
        &= \mathbf{T}_{r(j) r(j) - 1,\dots, m} \circ \mathbf{T}_{r(m), r(m) +
        1,\dots, m - 2} (E_{m - 1})
        \\
        &= \mathbf{T}_{r(j), r(j) - 1,\dots, m} (\mathbf{E}_{r(m), r(m) +
        1,\dots, m - 1})
        \\
        &= [\mathbf{E}_{r(m), r(m) + 1,\dots, m - 2}, \mathbf{E}_{r(j), r(j) -
        1,\dots, m - 1}]
        \\
        &= [X_{r(m), m - 1}, X_{m - 1, r(j) + 1}]
        \\
        &= \psi_7([X_{\ell, m - 1}, X_{i - 1, j}])
        \\
        &= \psi_7(X_{\ell, m - 1} X_{i - 1, j} - q^{-1}X_{i - 1, j}X_{\ell, m -
        1}).
    \end{align*}
    \noindent By the inductive hypothesis, we can replace $X_{\ell, m - 1}X_{i
    - 1, j}$ in the last line above with $qX_{i - 1, j}X_{\ell, m - 1} +
    (-q)^{m - r(m) - 1}X_{\ell j} - q\widehat{q} \sum_{r(m) < k < m - 1}
    (-q)^{m - k - 2} X_{kj}X_{\ell j}$. After making this substitution, the
    desired result follows.

    Since all of the relations given in the statement of this theorem hold in
    the algebra ${\mathcal U}_q(\mathfrak{n}_J)$, these relations are defining
    relations of ${\mathcal U}_q(\mathfrak{n}_J)$ by the PBW basis theorem
    \ref{PBW basis}.

\end{proof}

\end{document}